%% file: ms.tex
\documentclass[12pt]{amsart}

\usepackage{fancyhdr,lipsum}
\usepackage{amssymb,amsmath,graphicx,color,textcomp, amsthm,bbm,bbold, enumerate,booktabs, eufrak}
\usepackage{chngcntr}
\usepackage{algorithm,algorithmic}
\usepackage{apptools}
\AtAppendix{\counterwithin{theorem}{section}}
\definecolor{Red}{cmyk}{0,1,1,0}

\definecolor{verde}{cmyk}{1,0,1,0}

\definecolor{loka}{cmyk}{.5,0,1,.5}
\definecolor{azul}{cmyk}{1,1,0,0}

\usepackage{hyperref}

\numberwithin{equation}{section}

\newcommand{\eqd}{\stackrel{\tiny d}{=}}

\def\Ed{{\mathbb{E}}}
\def\Pd{{\mathbb{P}}}

\newcommand{\N}{\mathbb{N}}

\newcommand{\R}{\mathbb{R}}

\renewcommand{\P}{\mathbb{P}}

\renewcommand{\a}{\alpha}
\renewcommand{\b}{\beta}
\newcommand{\g}{\gamma}

\newcommand{\e}{\varepsilon}

\newcommand{\s}{\sigma}

\renewcommand{\l}{\lambda}

\newcommand{\be}{\begin{equation}}
\newcommand{\ee}{\end{equation}}

\newtheorem{theorem}{Theorem}
\newtheorem{proposition}{Proposition}
\newtheorem{definition}{Definition}
\newtheorem{lemma}{Lemma}
\newtheorem{corollary}[equation]{Corollary}

\newtheorem{obs}{Remark}

\title{Preferential Attachment Random Graphs with Edge-Step Functions}

\author{Caio Alves$^1$ }
\address{$^1$ Institute of Mathematics, University of Leipzig -- Augustusplatz 10, 04109, Leipzig,  Germany\newline
e-mail: {\itshape \texttt{caio.alves@math.uni-leipzig.de}}}

\author{Rodrigo Ribeiro$^{2}$}

\address{$^2$ IMPA, Estrada Da. Castorina, 110 CEP 22460-320 Rio de Janeiro, RJ, Brazil.
\newline
e-mail: {\itshape \texttt{rribeiro@impa.br}}}

\author{R{\'e}my Sanchis$^3$}
\address{$^3$Departamento de Matem{\'a}tica, Universidade Federal de Minas Gerais, Av. Ant\^onio
Carlos 6627 C.P. 702 CEP 30123-970 Belo Horizonte-MG, Brazil
\newline
e-mail: {\itshape \texttt{rsanchis@mat.ufmg.br}}}

\date{\today \\
    $^1$ Institute of Mathematics, University of Leipzig\\
    $^2$ IMPA, Instituto de Matem\'atica Pura e Aplicada. \\
	$^3$ Departamento de Matem{\'a}tica, Universidade Federal de Minas Gerais
}

\begin{document}

\begin{abstract} We propose a random graph model with preferential attachment rule and \textit{edge-step functions} that govern the growth rate of the vertex set. We study the effect of these functions on the empirical degree distribution of these random graphs. More specifically, we prove that when the edge-step function $f$ is a \textit{monotone regularly varying function} at infinity, the sequence of graphs associated to it obeys a power-law degree distribution whose exponent is related to the index of regular variation of $f$ at infinity whenever said index is greater than~$-1$. When the regularly variation index is less than or equal to~$-1$, we show that the proportion of vertices with degree smaller than any given constant goes to~$0$ a.\ s..
\end{abstract}
\maketitle
\section{Introduction}
\input{intro}

\section{Expected value analysis}\label{s:expdegdist}
\input{expectedanalysis}

\section{Concentration results for $\hat{P}_t(d,f)$}\label{sec:conc}

\input{concentrations}

\section{The case $\gamma \in [1,\infty)$}\label{sec:gamma1}

\input{roteiro-conv-qc.tex}

\section{Final Remarks}\label{sec:fr}

\input{finalremarks}

\appendix
\input{appendix}

{\bf Acknowledgements } 
C.A.  was supported by the Deutsche Forschungsgemeinschaft (DFG). R.R. was partially supported by Conselho Nacional de Desenvolvimento Cient\'\i fico e Tecnol\'{o}gico (CNPq).  R.S. has been partially supported by Conselho Nacional de Desenvolvimento Cient\'\i fico e Tecnol\'{o}gico (CNPq)  and by FAPEMIG (Programa Pesquisador Mineiro), grant PPM 00600/16. 
\bibliography{ref}
\bibliographystyle{plain}

\end{document}

%% file: intro.tex

In the late 1990s the seminal works of Strogatz and Watts \cite{SW98} and of \'Albert and Barab\'asi~\cite{BA99} brought to light two common features shared by real-life networks: \textit{small diameter} and \textit{power-law} degree distribution. In the first work the authors observed that large-scale networks of biological, social and technological origins presented diameters of much smaller order than the order of the entire network, a phenomenon they  called \textit{small-world}. In the second paper, the authors noted that the fraction of nodes having degree $d$ decays roughly as $d^{-\b}$ for some $\b > 1$, a feature known as \textit{scale-freeness}.

These findings motivated the task of proposing and investigating random graph models capable of capturing the two aforementioned features as well as other properties, such as large clique number \cite{ARS17} and large maximum degree \cite{M05}. The interested reader may be directed to \cite{CLBook, DBook, HBook} for a summary of rigorous results for many different models.

Usually the models proposed over the years are generative, in the sense that at each step~$t$ one obtain the random graph $G_t$ by performing some stochastic operation on $G_{t-1}$. In the well known Barab\'asi - \'Albert model \cite{BA99}, the stochastic operation consists of at each step a new vertex being added and a neighbor to it chosen among the previous vertices with probability proportional to its degree. This simple attachment rule, which is known as \textit{preferential attachment}, or PA-rule for short, is capable of producing graphs whose empirical degree distribution is well approximated by a power-law distribution with exponent $\b=3$. Many variants of the original B-A preferential attachment \cite{CLBook, CD18,DO14, KH02, MP14} have been introduced. These models also are capable of exhibiting power-law with different values of $\b$ and small-world phenomenon. 

In the remainder of this introduction we define our model in the next subsection and discuss the questions we have addressed in this paper. We end this section settling down some conventions and notations and explaining the paper's structure.

\subsection{The preferential attachment scheme with edge-step functions} The model we propose here has one parameter: a real non-negative function $f$ with domain given by the semi-line $[1, \infty)$ such that $||f||_{\infty} \le 1$. For the sake of simplicity, we start the process from an initial graph $G_1$ which is taken to be the graph with one vertex and one loop. We consider the two stochastic operations below that can be performed on any graph $G$: 
\begin{itemize}
	\item \textit{Vertex-step} - Add a new vertex $v$ and add an edge $\{u,v\}$ by choosing $u\in G$ with probability proportional to its degree. More formally, conditionally on $G$, the probability of attaching $v$ to $u \in G$ is given by
	\begin{equation}\label{def:PArule}
	P\left( v \rightarrow u \middle | G\right) = \frac{\mathrm{degree}(u)}{\sum_{w \in G}\mathrm{degree}(w)}.
	\end{equation}
	\item \textit{Edge-step} - Add a new edge $\{u_1,u_2\}$ by independently choosing vertices $u_1,u_2\in G$ according to the same rule described in the vertex-step. We note that both loops and parallel edges are allowed.
	
\end{itemize}
We consider a sequence $\{Z_t\}_{t\ge 1}$ of independent random variables such that $Z_t\eqd \mathrm{Ber}(f(t))$. We then define inductively a random graph process $\{G_t(f)\}_{t \ge 1}$  as follows: start with~$G_1$. Given $G_{t}(f)$, obtain $G_{t+1}(f)$ by either performing a \textit{vertex-step} on $G_t(f)$ when $Z_t=1$ or performing an \textit{edge-step} on $G_t(f)$ when $Z_t=0$.

We will call the function~$f$ by \textbf{\textit{edge-step function}}, though we follow an edge-step at time~$t$ with probability $1-f(t)$.

\subsection{Growth rate of the vertex-set}
For a fixed edge-step function $f$, our process generates a sequence of random (multi)graphs $\{G_t(f)\}_{t=1}^{\infty}$. The total vertices in $G_t(f)$ is also random and we let $V_t(f)$ denote this quantity.

Regarding the order of $G_t(f)$, in the vast majority of the preferential attachment random graph models it grows linearly with $t$, meaning that $V_t = \Theta(t)$, \textit{w.h.p} or deterministically depending on the model. For modeling purposes a sub-linear growth and some control over the growth rate of the vertex-set may be desirable, since in many real-world networks the rate of newborn nodes decreases with time while new connections continue be created with a high rate, e.g. Facebook and other social medias. In our setup, this may be achieved by choosing $f$ such that $f(t) \searrow 0$ as $t$ goes to infinity.

\subsection{The empirical degree distribution} Given a vertex $v$ in $G_t(f)$, we let $D_t(v)$ be its degree in $G_t(f)$. In this work we focus on the empirical degree distribution 
\begin{equation}
\hat{P}_t(d,f) := \frac{1}{V_t(f)}\sum_{v \in G_t(f)} \mathbb{1}\{D_t(v) = d\},
\end{equation}
i.e., the random proportion of vertices having degree $d$ in $G_t(f)$. for any $d \in \N$.

In many works, a combination of the preferential attachment rule (\ref{def:PArule}) with other attachment rules \cite{DO14,KH02,T15} proved themselves to be an efficient mechanism for generating graphs where~$\hat{P}_t(d)$ is essentially a powler-law distribution, meaning that, \textit{w.h.p}
\begin{equation}
\hat{P}_t(d) = Cd^{-\beta} \pm o(1),
\end{equation}
for some positive constant $C$ and some exponent $\beta$ generally lying in $(2,3]$. In \cite{CF03}, the authors investigated a very general model whose growth rule involves the case~$f(t) \equiv p$ and the possibility of choosing vertices uniformly instead of preferentially. Their model produces graphs whose empirical degree distribution follows a power-law distribution whose exponent lies in the range $(2,3]$. More specifically, in the particular case of~$f(t) \equiv p$, with $p \in (0,1)$, studied in \cite{CLBook, CF03}, the edge-step functions provided a control over the tail of the power-law distribution producing graphs obeying such laws with a tunable exponent~$\b = 2 +\frac{p}{2-p}$, i.e., they have shown that
\begin{equation}
\hat{P}_t(d) =  Cd^{-2-\frac{p}{2-p}} \pm O\left(\frac{d}{\sqrt{t}}\right),
\end{equation}
\textit{w.h.p}.
As pointed out in \cite{DEHH09}, it may be interesting to investigate models capable of generating graphs with $\beta$ lying in the range $(1,2]$. For instance, see \cite{DEHH09}, where the authors propose a model in which the number of edges added at each step is given by a sequence of independent random variables. This new rule is capable of reducing $\b$ but the vertex set still grows linearly in time, a property we would like to avoid in this paper.

In \cite{JM15}, the authors have introduced a generative model that combines the PA-rule with \textit{spatial proximity}, i.e., the vertices are added on some metric space and the closer the vertices are the more likely they are to become connected. In the paper the authors have addressed the characterization of the empirical degree distribution, proving that it is also well approximated by a power-law.

In our case, one of our results (Theorem~\ref{thm:powerlaw}), shows that for a broad class of functions
\[
\hat{P}_t(d,f) = Cd^{-2+\gamma} \pm O\left(\frac{d}{\sqrt{\int_1^tf(s)\mathrm{d}s}}\right)
\]
where $\gamma \in [0,1)$ depends only on the class $f$ belongs to.

\subsection{Our results}
Our main goal in this paper is to characterize $\hat{P}_t(\cdot,f)$ for a class of edge-step functions as general as possible. More precisely, we would like to obtain a very broad family $\mathfrak{F}$ of functions and a (generalized) distribution over the positive integers $(p(d))_{d \in \N}$ such that, for every fixed $d~\in\N$ and for all $f \in \mathfrak{F}$
\begin{equation}
\left| \hat{P}_t(d,f) - p(d)\right| \le o(1), 	
\end{equation}
with high probability.

The class we investigate here is the class of \textit{regularly varying} functions. A positive function~$f$ is said to be a regular varying function at infinity with \textit{index of regular variation} $\gamma$ if, for all~$a \in \mathbb{R}_+$, the identity below is satisfied
\begin{equation}
\lim_{t \to \infty}\frac{f(at)}{f(t)} = a^{\gamma}.
\end{equation}
In the particular case where $\gamma = 0$, $f$ is said to be a \textit{slowly varying} function. It will be useful to our purposes to recall that if $f$ is a regular varying function with index $\gamma$, the Representation Theorem (Theorem~\ref{thm:repthm}) assures the existence of a slowly varying function $\ell$ such that, for all $t$ in the domain of $f$, $f(t) = \ell(t)t^{\gamma}$.

For each $\gamma \in [0,\infty)$, we take the family $\mathfrak{F}$ to be the a subclass of all regular varying function of index $\gamma$, bounded by one and converging monotonically to zero. In notation, we will focus on functions belonging to the family defined below
\begin{equation}
\mathrm{RES}(-\gamma) := \left \lbrace f:[1, \infty] \longrightarrow [0,1] \; \middle | \; f  \textit{ is continuous, decreases to zero and has index } -\gamma\right \rbrace.
\end{equation}
The goal is to characterize $\hat{P}_t(\cdot,f)$ for all functions in $\mathrm{RES}(-\g)$, for all $\g \in [0,\infty)$. Our results establish a characterization for the empirical distribution depending only on the index $-\gamma$ and show a \textit{phase transition} on $\gamma $ equals $1$, meaning that for all $\gamma$ below this value,~$\hat{P}_t(d,f)$ is well approximated by a power-law whose exponent depends on $\gamma$ only, whereas for $\gamma \ge 1$ the empirical distribution vanishes for all fixed $d$. Specifically, if we let~$(p_{\gamma}(d))_{d \in \N}$ be the (generalized) distribution on $\N$ given by
\begin{equation}
p_{\gamma}(d) := \frac{(1-\g)\Gamma(2-\g)\Gamma(d)}{\Gamma(d+2-\gamma)},
\end{equation}
for $\gamma \in [0,1)$, a consequence of our results is that, for fixed~$d\in\N$, \textit{w.h.p},
\begin{equation}\label{eq:cor}
\left | \hat{P}_t(d,f) - p_{\g}(d) \right | \le o(1)
\end{equation}
for any $f\in \mathrm{RES}(-\g)$, with $\g \in [0,1)$. The error $o(1)$ on (\ref{eq:cor}) may depend on $f$ in an involved way and it is specified combining the estimates given by the two theorems below.

\begin{theorem}\label{thm:supE} Let $f \in \mathrm{RES}(-\g)$ with $\g \in [0,1)$ be such that~$f(t)=t^{-\gamma}\ell(t)$, where~$\ell$ is a slowly varying function. Then there exists a positive constant~$\tilde C = \tilde C(f)$, such that for every~$\alpha\in(0,1)$,
	\begin{equation}\label{eq:supE}
	\sup_{d \leq t} \left \lbrace \left| \Ed \left[V_t(f)\hat{P}_t(d,f)\right] - \Ed\left[V_t(f)\right] \cdot p_{\gamma}(d) \right| \right \rbrace \le \tilde C\cdot\mathrm{err}_t(\alpha,f),
	\end{equation}
	where $\mathrm{err}_t(\alpha,f)$ is defined as
	\begin{equation}
	\mathrm{err}_t(\a,f) := 1+\log t+\ell(t^{\alpha})t^{\alpha(1-\gamma)}+\frac{\ell(t)}{\ell(t^{\alpha})}t^{(1-\alpha)(1-\gamma)}+\sup_{s\geq t^\alpha}\mathcal{H}_{\ell,\gamma}(s)t^{1-\gamma}\ell(t)
	\end{equation}
	and $\mathcal{H}_{\ell,\gamma}(t)$ as the function of $t$ below
	\begin{equation}
	\mathcal{H}_{\ell,\gamma}(t):=\int_{0}^{1} \left| \frac{\ell(ut)}{\ell(t)}-1 \right|u^{-\gamma}\mathrm{d} u.
	\end{equation}
\end{theorem}
We stress the fact that the LHS of~\eqref{eq:supE} is~$o(\ell(t)t^{1-\gamma})=o(\Ed V_t(f))$, a fact implied later by Lemma~\ref{lemma:rvint}. This allows one to employ Theorem~\ref{thm:supE} to extract results about the behavior of the expected number of vertices having degree $d$ even when~$d$ is a function of~$t$ such that~$d=d(t)\xrightarrow{t\to\infty}\infty$, though the rate of growth of~$d(t)$ cannot be taken arbitrarily, but dependent on~$\mathcal{H}_{\ell,\gamma}$ and~$\gamma$. In fact, $d(t)=d_{\ell,\gamma}(t)$ should be chosen in such a way that
\begin{align}
\nonumber
	\lefteqn{\ell(t)^{-1}t^{-(1-\gamma)}\left(1+\log t+\ell(t^{\alpha})t^{\alpha(1-\gamma)}+\frac{\ell(t)}{\ell(t^{\alpha})}t^{(1-\alpha)(1-\gamma)}+\sup_{s\geq t^\alpha}\mathcal{H}_{\ell,\gamma}(s)t^{1-\gamma}\ell(t)\right)}\phantom{****************************************}
	\\ \label{eq:maxdtreq}
	&=o(d_{\ell,\gamma}(t)^{-(2-\gamma)}).
	\end{align}
Given~$c\in\R$, $\delta\in(0,1)$, we consider the functions that to each~$t>1$ associate respectively
	\begin{equation}
	\label{eq:svexamples}
	c,\quad\quad (\log t)^c, \quad\quad \log\log t,\quad\quad \exp( (\log t)^\delta).
	\end{equation}
	For the specific slowly varying functions in~\eqref{eq:svexamples}, one can see by Remark~\ref{obs:rateH}, \eqref{eq:Hexamples} and elementary asymptotic analysis,
	\begin{align*}
	d_{c,\gamma}(t) &\quad\text{ can be chosen in }\quad o\left(t^{\frac{1-\gamma}{2(2-\gamma)}}\right);
	\\
	d_{\log^c,\gamma}(t) &\quad\text{ can be chosen in }\quad o\left((\log t)^{\frac{1}{2-\gamma}}\right);
	\\
	d_{\log\log,\gamma}(t) &\quad\text{ can be chosen in }\quad o\left((\log t\cdot \log\log t)^{\frac{1}{2-\gamma}}\right);
	\\
	d_{\exp(\log^\delta),\gamma}(t) &\quad\text{ can be chosen in }\quad o\left((\log t)^{\frac{1-\delta}{2-\gamma}}\right).
	\end{align*}

The previous theorem assures us that $\Ed [V_t(f)\hat{P}_t(d,f)]/\Ed V_t(f)$ is close to $p_{\g}(d)$. The next one assures that $\hat{P}_t(d,f)$ is concentrated around $\Ed [V_t(f)\hat{P}_t(d,f)]/\Ed V_t(f)$.
\begin{theorem}\label{thm:powerlaw} Let $f \in \mathrm{RES}(-\g)$ with $\g \in [0,1)$. Then, for all $d \in \N$ and $A>0$ such that 
	\begin{equation}\label{def:conditionA}
	A<\frac{1}{4d\log(t)}\sqrt{\frac{\Ed V_t(f)}{1-\gamma} },
	\end{equation} 
	we have
	\begin{equation}
	\left| \hat{P}_t(d,f) - \frac{\Ed \left[V_t(f)\hat{P}_t(d,f)\right]}{\Ed V_t(f)}\right| \le A\cdot\frac{10d}{\sqrt{(1-\gamma)\Ed V_t(f)}}, 
	\end{equation}
	with probability at least $1-3e^{-A^2/3}$. 
\end{theorem}
The power-law degree distribution of the random graph when $f\in\mathrm{RES}(-\gamma)$ is provided by Theorems~\ref{thm:supE} and~\ref{thm:powerlaw}, and is formally stated in the corollary below.
\begin{corollary}[Power-law degree distribution]Let $f \in \mathrm{RES}(-\g)$ with $\g \in [0,1)$. Then, for all $d \in \N$, $\alpha\in(0,1)$, and $A>0$ satisfying (\ref{def:conditionA}), 
	\begin{equation}
	\left| \hat{P}_t(d,f) - \frac{(1-\g)\Gamma(2-\g)\Gamma(d)}{\Gamma(d+2-\gamma)}\right| \le A\sqrt{\frac{40d^2}{\Ed V_t(f)}} + \mathrm{err}_t(\alpha,f), 
	\end{equation}
	with probability at least  $1-3e^{-A^2/3}$. 
\end{corollary}
We must stress out the fact that $(p_{\g}(d))_{d \in \N}$ is a generalized distribution for $\gamma \in (0,1)$. In this regime, we have mass escaping to infinity, due, possibly, to the existence of vertices of very high degree (c.f. Section~\ref{sec:fr} for a discussion about the maximum degree). On the other hand, what may be surprisingly is the fact that $G_t(f)$ has mean degree of order $t^{\gamma}$, \textit{w.h.p}, but still has positive proportion of vertices of constant degree.

We also point out that, another byproduct of our theorems is that for all $d \in \N$, we have that
\[
\lim_{t \to \infty} \hat{P}_t(d,f) = p_{\gamma}(d), \textit{ a.s}
\]
for any $f \in \mathrm{RES}(-\gamma)$,  with $\g \in [0,1)$.

For functions whose index $-\gamma$ lies on $(-\infty, -1]$, all the mass of the empirical degree distribution escapes to infinity in the sense that the fraction of vertices having degree $d$ goes to zero for any value of $d$. 
\begin{theorem}[All mass escapes to infinity]\label{t:notpowerlaw}Let $f \in \mathrm{RES}(-\g)$ with $\gamma \ge 1$. Then, for all fixed~$ d \in \N$,
	\begin{equation}
	\lim_{t \to \infty} \hat{P}_t(d,f) = 0, \textit{ a.s.}
	\end{equation}
\end{theorem}
\subsection{Notation and conventions}
\subsubsection{General}Regarding constants, we let $C,C_1,C_2,\dots$ and $c,c_1,c_2,\dots$ be positive real numbers that do not depend on~$t$ whose values may vary in different parts of the paper. The dependence on other parameters will be highlighted throughout the text.

Since our model is inductive, we use the notation $\mathcal{F}_t$ to denote the $\s$-algebra generated by all the random choices made up to time $t$. We then have the natural filtration $\mathcal{F}_1 \subset \mathcal{F}_2 \subset \dots $ associated to the process.
\subsubsection{Graph theory} We abuse the notation and let~$V_t(f)$ denote the set and the number of vertices in~$G_t(f)$. Given a vertex~$v\in V_t(f)$, we will denote by~$D_t(v)$ its degree in~$G_t(f)$. We will also denote by~$\Delta D_t(v)$ the \emph{increment} of the discrete function~$D_t(v)$ between times~$t$ and~$t+1$, that is,
\begin{equation}
\Delta D_t(v) =D_{t+1}(v) -D_t(v).
\end{equation}
For every $d \in \N$ and edge-step function $f$, we let $N_t(d,f)$ be the number of vertices of degree ~$d$. Naturally, $N_t(\le d,f)$ stands for the number of vertices having degree at most $d$. Our empirical degree distribution is written as
\begin{equation}
\hat{P}_t(d,f) = \frac{N_t(d,f)}{V_t(f)}.
\end{equation}
Since the expected number of vertices appear repeatedly throughout the paper we reserve a special notation for it
\begin{equation}
F(t) := \Ed V_t(f).
\end{equation} 
We also drop the dependency on $f$ on all the above notations when the function $f$ is clear from the context or when we are talking about these observables in a very general way, including in other preferential attachment models.
\subsubsection{Asymptotic} We will make use of asymptotic notation $o$ and $O$, which will presuppose asymptotic in the time parameter $t$, except when another parameter is explicitly indicated. We also use the notation $O_{d}$ indicating that the implied constant depends only on the quantity $d$. For instance, $O_d(t)$ denotes a quantity bounded by~$t$ times a constant depending only on $d$. Moreover, for any two sequence of real numbers $(a_n)_{n \in \N}$ and $(b_n)_{n \in \N}$, we write~$a_n \approx b_n$, if~$a_n/b_n$ converges to a non-zero constant. We write $a_n \sim b_n$ for the particular case $c=1$.
\subsection{Organization} Section~\ref{s:expdegdist} is devoted to the analysis of expected number of vertices having degree $d$ for $f \in \mathrm{RES}(-\gamma)$ with $\gamma \in [0,1)$ proving Theorem~\ref{thm:supE}. In Section~\ref{sec:conc} we prove a general concentration result for $\hat{P}_t(d,f)$ which holds for any edge-step function $f$. Then, we use this general result exploiting our knowledge about $f$ when it belongs to the class $\mathrm{RES}(-\g)$, for $\g \in [0,1)$, to prove Theorem~\ref{thm:powerlaw}. The case $\g \ge 1$ is treated separately in Section~\ref{sec:gamma1}, where we prove Theorem~\ref{t:notpowerlaw} and all the results needed. We end this paper presenting in Section~\ref{sec:fr} some brief discussion about the \textit{affine case of this model} and the \textit{maximum degree}. In the first topic we show that the presence of edge-step functions inhibits the effect of constant terms added to the rule (\ref{def:PArule}). In the second topic we provide some computations that indicate that the order of the maximum degree also varies according to how fast the edge-step goes to zero. For $f$ whose index of regular variation $-\g$ lies on $(-1,0)$ a maximum degree of order $t$ seems to be achieved, whereas the case where $f$ is slowly varying seems to be richer in the sense that the order of the maximum degree at time $t$ may depend on $f$.

%% file: expectedanalysis.tex
In this section, we prove Theorem~\ref{thm:supE}, which gives us estimates on the expected number of vertices having degree exactly $d$ for $f \in \mathrm{RES}(-\gamma)$, with $\g \in [0,1)$. Our first result in this direction is the following recurrence relation for $\Ed N_t(d,f)$ which holds for any edge-step function~$f$.
\begin{lemma}\label{lemma:recurntd} Let $\Ed N_t(d)$ denote $\Ed N_t(d,f)$ for a fixed edge-step function $f$. Then, $\Ed N_t(d)$ satisfies
\begin{equation}\label{eq:end1}
\Ed N_{t+1}(1) = \left(1 - \frac{2-f(t+1)}{2t} + \frac{(1-f(t+1))}{4t^2}\right)\Ed N_{t}(1) + f(t+1) ,
\end{equation} 
and for a fixed integer~$d\ge 2$,
\begin{align} 
\nonumber
\Ed N_{t+1}(d) &= \left(1 - \frac{(2-f(t+1))d}{2t} + \frac{(1-f(t+1))d^2}{4t^2}\right)\Ed N_{t}(d) 
\\ \label{eq:end}
&\quad+ \left(\frac{(2-f(t+1))(d-1)}{2t} - \frac{(1-f(t+1))(d-1)^2}{4t^2}\right)\Ed N_{t}(d-1) 
\\  \nonumber
&\quad+ \frac{(1-f(t+1))(d-2)^2}{4t^2}\Ed N_{t}(d-2).
\end{align}
\end{lemma}
\begin{proof}
	There are two possible ways in which a vertex~$v$ increases its degree by~$1$ at time~$t+1$: either a vertex is created at time~$t+1$ and connects to~$v$, or an edge is created instead and exactly one of its endpoints connects to~$v$. This implies
	\begin{equation}\label{eq:vardeg1}
	\begin{split}
	\mathbb{P}\left(\Delta D_t(v) = 1 \middle | \mathcal{F}_t \right) &=  f(t+1)\frac{D_t(v)}{2t} +2(1-f(t+1))\frac{D_t(v)}{2t}\left( 1 - \frac{D_t(v)}{2t}\right)\\
	& = \left(1-\frac{f(t+1)}{2}\right)\frac{D_t(v)}{t} - 2\left(1-f(t+1)\right)\frac{D^2_t(v)}{4t^2}.
	\end{split}
	\end{equation}
	In order for the degree of~$v$ to increase by~$2$ at time~$t+1$ the only possibility is that an edge step occurs and both endpoints of the new edge are attached to~$v$, creating a loop. This implies
	\begin{equation}\label{eq:vardeg2}
	\mathbb{P}\left(\Delta D_t(v) = 2 \middle | \mathcal{F}_t \right) =  \left(1-f(t+1)\right)\frac{D^2_t(v)}{4t^2}. 
	\end{equation}
	We may write $N_{t+1}(d)$ as
	\begin{align}
	\nonumber
	\lefteqn{N_{t+1} ( d)} \phantom{***}\\ 
	\label{eq:recuntd}
	&= \sum_{\substack{v\in V_t(f) \\ D_t(v)=d}} \mathbb{1}{\{ \Delta D_t (v) = 0
		\}} + \sum_{\substack{v\in V_t(f) \\ D_t(v)=d-1}} \mathbb{1}{\{ \Delta D_t (v) = 1 \}} + \sum_{\substack{v\in V_t(f) \\ D_t(v)=d-2}} \mathbb{1}{\{ \Delta D_t (v) = 2 \}}.
	\end{align}
	Combining the three above equations and taking the expected value on (\ref{eq:recuntd}), we obtain~(\ref{eq:end}). For the case $d=1$, just observe that
	\[
	N_{t+1} (1) = \sum_{\substack{v\in V_t(f) \\ D_t(v)=1}} \mathbb{1}{\{ \Delta D_t (v) = 0
		\}} +  \mathbb{1}{\{\text{a vertex born at time } t+1  \}}.
	\]
\end{proof}
From now on we restrict our edge-step functions to the class $\mathrm{RES}(-\g)$ with $\g$ always in the range $[0,1)$. Note that in this case, by the Representation Theorem (Theorem~\ref{thm:repthm}), there exists a \textit{slowly varying} function $\ell$ such that 
\begin{equation}\label{eq:lf}
f(t) = t^{-\g}\ell(t),
\end{equation} 
for all $t$. Before proving Theorem~\ref{thm:supE}, we introduce notation and state a crucial lemma about regularly varying functions and their sums.
\begin{lemma}[Proof in Appendix~\ref{app:rvf}]
\label{lemma:rvint}
Let~$\gamma\in[0,1)$ and let~$\ell:\R\to\R$ be a continuous slowly varying function such that~$s\mapsto\ell(s)s^{-\gamma}$ is non-increasing. Define
\begin{equation}
\label{eq:rvinterrordef}
\mathcal{H}_{\ell,\gamma}(t):=\int_{0}^{1} \left| \frac{\ell(ut)}{\ell(t)}-1 \right|u^{-\gamma}\mathrm{d} u.
\end{equation}
Then~$\mathcal{H}_{\ell,\gamma}(t)$ is well defined and the following holds
\begin{itemize}
	\item[(i)] $\displaystyle \mathcal{H}_{\ell,\gamma}(t)\xrightarrow{t\to\infty}0 ;$ \\
	\item [(ii)] $\displaystyle  \mathcal{G}_{\ell,\gamma}(t):=\left| \sum_{k=1}^{t}\ell(k)k^{-\gamma}  -\frac{t^{1-\gamma}\ell(t)}{1-\gamma}     \right| \left(  t^{1-\gamma}\ell(t)  \right)^{-1} \leq \mathcal{H}_{\ell,\gamma}(t)+ \left(  t^{1-\gamma}\ell(t)   \right)^{-1} . $ 
\end{itemize}
\end{lemma}
\begin{obs}
\label{obs:rateH} Here we provide some examples of the kind of rate of decay that is associated to the above Lemma. Consider the functions defined in~\eqref{eq:svexamples}. Elementary calculations then show that their associated error terms are, respectively,
\begin{equation*}
\mathcal{H}_{c,\gamma}(t)=0,\quad\mathcal{H}_{(\log t)^c,\gamma}(t)=O((\log t)^{-1}),\quad
\end{equation*}
\begin{equation} 
\label{eq:Hexamples}\mathcal{H}_{\log\log t,\gamma}(t)=O((\log t\log\log t)^{-1}),\quad
\mathcal{H}_{\exp( \log^\delta t),\gamma}(t)=O((\log t)^{-(1-\delta)}).
\end{equation}
\end{obs}
Now we have all the tools needed for the proof of Theorem~\ref{thm:supE}. The proof is inspired by~\cite{HBook} Section~$8.6.2$, though our context prevents a straightforward application. The essential idea is that~$N_t(d)$ and~$p_\gamma(d)F(t)$ satisfy very similar recurrence relations in~$d$ when~$t$ is large. Quantifying this similarity allows us to prove that they are indeed close as sequences in~$d$ in the~$L_\infty(\N)$ sense. We expand on this idea below.  
\begin{proof}[Proof of Theorem \ref{thm:supE}]
For each~$t\geq 2$ we define the linear operator
\[T_{t}:L_\infty(\N)\to L_\infty(\N) \]
that maps each bounded sequence~$(a_j)_{j\geq 1}$ to a sequence defined by
\begin{align}
	\label{eq:Tdef}
	(T_{t}((a_j)_{j\geq 1}))_k&:=
	\left(1-\frac{2-f(t)}{2(t-1)}k+\frac{1-f(t)}{4(t-1)^2}k^2\right)a_{k}
	\\ \nonumber
	&\quad +
	\left(\frac{2-f(t)}{2(t-1)}(k-1)-\frac{2(1-f(t))}{4(t-1)^2}(k-1)^2\right)a_{k-1}\mathbb{1}\{k>1\}
	\\ \nonumber
	&\quad +
	\frac{1-f(t)}{4(t-1)^2}(k-2)^2 a_{k-2}\mathbb{1}\{k>2\}.
	\end{align}
	Since the coefficients of~$a_k$, $a_{k-1}$, and~$a_{k-2}$ above are all nonnegative, we get
	\begin{align}
	\|(T_{t}((a_j)_{j\geq 1}))\|_\infty \nonumber
	&\leq
	\sup_k\Bigg(
	\left(1-\frac{2-f(t)}{2(t-1)}k+\frac{1-f(t)}{4(t-1)^2}k^2\right)\|(a_j)_{j\geq 1}\|_\infty
	\\ \label{eq:Tcontraction}
	&\quad\quad\quad\quad +
	\left(\frac{2-f(t)}{2(t-1)}(k-1)-\frac{2(1-f(t))}{4(t-1)^2}(k-1)^2\right)\|(a_j)_{j\geq 1}\|_\infty
	\\ \nonumber
	&\quad\quad\quad\quad +
	\frac{1-f(t)}{4(t-1)^2}(k-2)^2 \|(a_j)_{j\geq 1}\|_\infty
	\Bigg)
	\\ \nonumber
	&\leq
	\left(
	1 -  \frac{2-f(t)}{2(t-1)}+\frac{1-f(t)}{2(t-1)^2}
	\right)  \|(a_j)_{j\geq 1}\|_\infty,
	\end{align}
	which implies~$T_t$ is a contraction on~$L_\infty(\N)$. Furthermore, by Lemma~\ref{lemma:recurntd}, we have
	\[
	\Ed[N_t(d)]=(T_t((\Ed[N_{t-1}(k)])_{k\geq 1}))_{d}+f(t)\cdot\mathbb{1}\{d=1\}.
	\]
	Our goal is to use~$T_t$ to bound the distance between the sequence of expectations above and the sequence $(F(t)\cdot p_{\gamma}(d))_{d\geq 1}$. We will do so by showing that~$(F(t)\cdot p_{\gamma}(d))_{d\geq 1}$ is very close to being a fixed point of another operator defined below in~\eqref{eq:Sdef}, this operator being itself very close to~$T_t$ for large~$t$. 
	
	By elementary properties of the Gamma function, we see that~$(p_{\gamma}(d))_{d\geq 1}$ is defined recursively by
	\begin{equation}
	\label{eq:mdef}
	p_{\gamma}(d)=\frac{d-1}{d+1-\gamma}p_{\gamma}(d-1);\quad\quad p_{\gamma}(1)=\frac{1-\gamma}{2-\gamma}.
	\end{equation}
	By Lemma~$\ref{lemma:rvint}$ we have
	\begin{align}
	\frac{F(t-1)}{F(t)}&=1-\frac{f(t)}{F(t)}
	\label{eq:FRES}
	= 1-\frac{t^{-\gamma}\ell(t)}{(1+O(\mathcal{G}_{\ell,\gamma}(t)))\frac{t^{1-\gamma}\ell(t)}{1-\gamma}}
	= 1-\frac{1-\gamma}{t}(1+O(\mathcal{G}_{\ell,\gamma}(t))).
	\end{align}
	Observe that the sequence $(\Ed N_t(d))_{d\ge 1}$ has all its coordinates, for $d >2t$, equal zero. Therefore, we must truncate the sequence~$(p_{\gamma}(d))_{d\geq 1}$ for~$d>2t$ obtaining the sequence $(m_{d,t})_{d\geq 1}$ defined by
	\[
	m_{d,t} :=p_{\gamma}(d)\mathbb{1}\{d \leq 2t\}.
	\]
	Now consider the operator~$S_t:\R^\N\to\R^\N$ defined by
	\begin{equation}
	\label{eq:Sdef}
	(S_t(a_j)_{j\geq 1})_d:=\left(\frac{d-1}{1-\gamma}a_{d-1}-\frac{d}{1-\gamma}a_{d}\right)\mathbb{1}\{d \leq t\},
	\end{equation}
	and note that, by an application of (\ref{eq:mdef}), the sequence $(m_{d,t})_{d\ge 1}$ satisfies
	\begin{equation}
	\label{eq:SMd}
	m_{d,t}=(S_t(m_{j,t})_{j\geq 1})_d+\mathbb{1}\{d=1\}.
	\end{equation}
	Defining then
	\begin{equation}
	\mathcal{E}_d(t):=((f(t)S_t+F(t-1)(I-T_{t}))(m_{j,t})_{j\geq 1})_d,
	\end{equation}
	where~$I$ denotes the identity operator in~$\R^\N$, we get
	\begin{align}
	\nonumber
	F(t)m_{d,t} 
	&=
	F(t-1)m_{d,t}+f(t)m_{d,t}
	\\ \nonumber
	&=
	F(t-1)m_{d,t}+f(t)(S_t(m_{j,t})_{j\geq 1})_d+f(t)\mathbb{1}\{d=1\}
	\\ \label{eq:Ftmais1Md}
	&=(T_{t}(	F(t-1)m_{j,t})_{j\geq 1})_d+f(t)\mathbb{1}\{d=1\}+\mathcal{E}_d(t).
	\end{align}
	We will now bound from above the terms in~$(\mathcal{E}_d(t))_{d\geq 1}$, which will be the main error terms associated to the approximation of~$\Ed[N_t(d)]$ by~$F(t) m_{d,t}$. Note that~$\|(m_{j,t})_{j\geq 1}\|_{\infty}\leq 1$ and~$\sup_d d^{2-\gamma}p_{\gamma}(d)<\infty$, which together with~\eqref{eq:Tdef} imply
	\begin{align}
	\label{eq:TmenosI}
	((T_{t}-I)((m_{j,t})_{j\geq 1}))_d&=
	-\frac{d}{t-1} m_{d,t}
	+
	\frac{d-1}{t-1}(d-1)m_{d-1,t}
	+O(f(t)t^{-1}+d^\gamma t^{-2}\mathbb{1}\{d\leq 2t\}).
	\end{align}
	Note that the function represented by the~$O$ notation above is actually~$o(t^{-1})$, since $f$ decreases to zero. We then obtain, by~$ (\ref{eq:lf},\ref{eq:mdef},\ref{eq:FRES}) $, for~$d\leq 2t$, 
	\begin{align}
	\nonumber
	\mathcal{E}_d(t)
	&=	F(t-1)\left(\frac{d}{t-1} m_{d,t}
	-
	\frac{d-1}{t-1}m_{d-1,t}\right)
	+
	f(t)\left(-\frac{d}{1-\gamma} m_{d,t}+
	\frac{d-1}{1-\gamma}m_{d-1,t}\right)+o(t^{-1})
	\\ \nonumber
	&=\frac{d}{1-\gamma}p_{\gamma}(d)\left(  \frac{(1+O(\mathcal{G}_{\ell,\gamma}(t-1)))\ell(t-1)(t-1)^{1-\gamma}}{(t-1)}  -\ell(t)t^{-\gamma}  \right)
	\\ \label{eq:boundedt}
	&\quad+
	\frac{(d-1)}{1-\gamma}\frac{d+1-\gamma}{d-1}p_{\gamma}(d)\left( \ell(t)t^{-\gamma}   -  \frac{(1+O(\mathcal{G}_{\ell,\gamma}(t-1)))\ell(t-1)(t-1)^{1-\gamma}}{(t-1)}\right)
	+o(t^{-1})
	\\ \nonumber
	&=
	t^{-\gamma}p_{\gamma}(d)\left(\ell(t)-      (1+O(\mathcal{G}_{\ell,\gamma}(t-1)))\ell(t-1)\left(1-\frac{\gamma}{t}+O(t^{-2})\right)       \right) +o(t^{-1})
	\\ \nonumber
	&= t^{-\gamma}p_{\gamma}(d)(\ell(t)-\ell(t-1))+t^{-\gamma}\ell(t-1)O(\mathcal{G}_{\ell,\gamma}(t-1))+o(t^{-1}).
	\end{align}
	Furthermore, $\mathcal{E}_d(t)=0$ for~$d>2t$. The above equation together with~\eqref{eq:Tcontraction} and~\eqref{eq:Ftmais1Md} implies
	\begin{align}
	\nonumber
	\|(\Ed[N_t(d)] -  F(t)m_{d,t})_{d\geq 1} \|_\infty
	&\leq 
	\|T_t((\Ed[N_{t-1}(d)]  - F(t-1)m_{d,t})_{d\geq 1})\|_\infty+ \|(\mathcal{E}_d(t))_{d \geq 1}\|_\infty
	\\ \nonumber
	&\leq \|(\Ed[N_{t-1}(d)]  - F(t-1)m_{d,t-1})_{d\geq 1}\|_\infty
	\\ \label{eq:Tmbound1}
	&\quad+ \|(\mathcal{E}_d(t))_{d \geq 1}\|_\infty+F(t-1)p_{\gamma}(t)
	\\ \nonumber
	&\leq C + \sum_{s=1}^{t}  \left(\|(\mathcal{E}_d(s))_{d \geq 1}\|_\infty+F(s-1)p_{\gamma}(s) \right),
	\end{align}
	since 
	\[
	\|(\Ed[N_1(d)] -  F(1)m_{d,1})_{d\geq 1} \|_\infty < C.
	\]
for some constant~$C>0$. Since~$p_{\gamma}(s)=O(s^{-2+\gamma})$, Lemma~\ref{lemma:rvint} implies
\[
\sum_{s=1}^{t}F(s-1)p_{\gamma}(s) \leq C \sum_{s=1}^{t} s^{-1} \leq C\log t,
\]
and the proof will be finished once we  show an upper bound for~$\sum_{s=1}^{t}\|(\mathcal{E}_d(s))_{d \geq 1}\|_\infty$ of the desired order. Since~$\ell(s)s^{-\gamma}$ is decreasing, we get
\begin{align}
\nonumber
\sum_{s=1}^{t} s^{-\gamma}|\ell(s)-\ell(s-1)|
&\leq
C+\int_{1}^{t} s^{-\gamma}|\ell(s)-\ell(t)|\mathrm{d} s+ \int_{1}^{t} s^{-\gamma}|\ell(s-1)-\ell(t)|\mathrm{d} s
\\ \label{eq:Tmbound2}
&\leq 
C+\ell(t)t^{1-\gamma}\mathcal{H}_{\ell,\gamma}(t)+\ell(t)t^{1-\gamma}\int_{0}^{t-1}\left|\frac{\ell(y)}{\ell(t)}-1\right|\frac{(y+1)^{-\gamma}}{t^{1-\gamma}}\mathrm{d} y
\\ \nonumber
&\leq C+2\ell(t)t^{1-\gamma}\mathcal{H}_{\ell,\gamma}(t).
\end{align}
By Lemma~\ref{lemma:rvint}, we have
\begin{align}
\nonumber
\sum_{s=1}^t s^{-\gamma}\ell(s-1)\mathcal{G}_{\ell,\gamma}(s-1)
&=
\sum_{s=1}^{t^\alpha} s^{-\gamma}\ell(s-1)\mathcal{G}_{\ell,\gamma}(s-1)
+\sum_{s=t^\alpha+1}^{t} s^{-\gamma}\ell(s-1)\mathcal{G}_{\ell,\gamma}(s-1)
\\ \label{eq:Tmbound3}
&\leq C\ell(t^{\alpha})t^{\alpha(1-\gamma)}+C\sup_{s\geq t^\alpha}\mathcal{G}_{\ell,\gamma}(s)t^{1-\gamma}\ell(t)
\\ \nonumber
&\leq C\left(\ell(t^{\alpha})t^{\alpha(1-\gamma)}+\sup_{s\geq t^\alpha}\mathcal{H}_{\ell,\gamma}(s)t^{1-\gamma}\ell(t)    +\frac{\ell(t)}{\ell(t^{\alpha})}t^{(1-\alpha)(1-\gamma)}   \right)
\end{align}
Together with~$(\ref{eq:Tmbound1},\ref{eq:Tmbound2})$ this implies
\begin{align*}
\lefteqn{\|(\Ed[N_t(d)] -  F(t)m_{d,t})_{d\geq 1} \|_\infty}\phantom{*******}
\\&\leq
C\left(1+\log t+\ell(t^{\alpha})t^{\alpha(1-\gamma)}+\frac{\ell(t)}{\ell(t^{\alpha})}t^{(1-\alpha)(1-\gamma)}+\sup_{s\geq t^\alpha}\mathcal{H}_{\ell,\gamma}(s)t^{1-\gamma}\ell(t)\right)
,
\end{align*}
finishing the proof of the result.
\end{proof}

%% file: concentrations.tex
We begin this section with a brief discussion about why the presence of the edge-step function requires concentration results sharper than those found in the present literature.

For general concentration results for~$N_t(d)$ the usual approach is to obtain a (sub, super)martingale involving $N_t(d)$, then to prove that it has bounded increments and finally to apply Azuma's inequality (Theorem~\ref{t:azuma}). These (sub, super)martingales are usually $N_t(d)$ properly normalized or the Doob martingale, see \cite{CLBook, HBook} for the two distinct approaches. This sort of argument leads to concentration result for~$N_t(d)$ with a deviation from its mean typically of order~$\sqrt{t}$. More precisely, it is proven that
\begin{equation}\label{eq:approxPowerlaw}
N_t(d) \sim \Ed \left[N_t(d)\right] \pm A\sqrt{t},
\end{equation}
with high probability, and from the analysis of the expected value it comes that 
\begin{equation}\label{eq:approxExpec}
\Ed \left[N_t(d)\right] \sim \frac{t}{d^{\beta}},
\end{equation}
where $\beta$ is the power-law exponent. Since the edge-step function controls the growth rate of the vertex set, in the presence of a regular varying edge-step function the expected value of~$N_t(d,f)$ analysis leads to
\begin{equation}
\Ed \left[N_t(d,f)\right] \sim \frac{\int_1^tf(s)ds}{d^{\beta}}.
\end{equation}
On the other hand, a straightforward application of the usual approach would give us
\begin{equation}
\frac{\int_1^tf(s)ds}{d^{\beta}} - A\sqrt{t} \le N_t(d,f) \le \frac{\int_1^tf(s)ds}{d^{\beta}} + A\sqrt{t},
\end{equation}
with high probability. However, this is trivially true for some choices of $f$, e.\ g.\ if $f \in \mathrm{RES}(-\g)$ with~$\gamma>1/2$. This issue demands a result finer than those found in the literature, at least for a particular class of functions. We overcome it by applying Freedman's inequality (Theorem~\ref{teo:freedman}) instead of Azuma's. Freedman's inequality takes into account our knowledge about the past of the martingale to estimate its increments instead of simply bounding them deterministically as it is done in Azuma's. However, Freedman's inequality requires upper bounds on the conditional quadratic variation of the martingale~(see~(\ref{def:quadvar})), which may be more involved than obtaining deterministic bounds for the increments.

For a fixed time $t\ge 1 $, $d \in \N$ and \emph{any} edge-step function $f$, we define the following sequence of random variables
\begin{equation}\label{def:M}
M_s(d,f) := \Ed \left[ N_t(d,f)\middle | \mathcal{F}_s\right].
\end{equation}
Since the degree $d$ and the edge-step function $f$ will be fixed for the remainder of this section, we will omit the dependency on them, denoting simply $\{M_s\}_{s\ge 1}$ when there is no risk of confusion. Observe that by the tower property of the conditional expected value, it follows that $\{M_s\}_{s \ge 1}$ is a martingale.

We will obtain our concentration result applying Freedman's inequality (Theorem \ref{teo:freedman}) to~$M_t$. It requires estimates on the increments of $\{M_s\}_{s\ge 1}$ as well as on its conditional quadratic variation, see \ref{def:quadvar}. We begin by showing that $\{M_s\}_{s\ge 1}$ is actually a bounded increment martingale, which is done in the next lemma. Since it is almost in line with proof of Lemma 8.6 in~\cite{HBook}, we skip some details throughout the proof.
\begin{lemma}[Bounded increments]\label{lemma:boundincr} Let $\{M_s\}_{s\ge 1}$ be as in (\ref{def:M}). Then, it satisfies
	\begin{equation}
	\left | M_{s+1} - M_s \right | \le 4,
	\end{equation}
	for all values of $s$.
\end{lemma}
\begin{proof} For a fixed $s$, consider in the same probability space the process 
	\[
	\{G'_r(f)\}_{r\ge1}\eqd\{G_r(f)\}_{r\ge1},
	\]
	which evolves following exactly the steps of~$\{G_r(f)\}_{r\ge1}$ for all $r\le s$ and then evolves independently for $r\ge s+1$. Let $\{\mathcal{F}'_{r}\}_{r\ge 1}$ be the natural filtration associated to the prime process.
	
	Denote by~$v_r$ the vertex born at time~$r\geq 1$, and recall the the definition of~$(Z_r)_{r\geq 1}$, the Bernoulli variables that control whether a vertex or edge step was taken at each time. Observe that we may write $N_t( d,f)$ as
	\begin{equation}
	N_t(d,f) = \sum_{r=1}^t \mathbb{1}\{D_t(v_r)=  d\}Z_r,
	\end{equation}
	consequently, we may express $\Delta M_s$ as
	\begin{eqnarray}
	M_{s+1} - M_s = \sum_{r=1}^t \Pd\left( D_t(v_r) = d, Z_r =1 \middle | \mathcal{F}_{s+1}\right) - \Pd\left( D_t(v_r) = d,Z_r =1 \middle | \mathcal{F}_{s}\right).
	\end{eqnarray}
	Let $D'_t(v_r)$ and $Z_r'$ denote the counterpart to $D_t(v_r)$ and $Z_r$ in the prime process respectively and note that 
	\begin{equation}
	\Pd\left( D'_t(v_r) = d,Z'_r=1 \middle | \mathcal{F}_{s}\right) = \Pd\left( D'_t(v_r) = d,Z'_r=1 \middle | \mathcal{F}_{s+1}\right),
	\end{equation}
	since $\mathcal{F}_{s+1}$ is $\mathcal{F}_s$ (which is equal to $\mathcal{F}'_s$) with independent information from $D'_t(i)$ and $Z'_r$ added. Moreover, since the evolution of each vertex's degree only depends on itself, we also have
	\begin{equation}
	\Pd\left( D_t(v_r) =  d,Z_r=1 \middle | \mathcal{F}_{s+1}\right) = \Pd\left( D_t(v_r) =d, Z_r=1 \; \middle | \; D_{s+1}(v_r)\right)
	\end{equation}
	and
	\begin{equation}
	\begin{split}
	\Pd\left( D'_t(v_r) = d, Z'_r=1 \middle | \mathcal{F}_{s+1}\right) &= \Ed \left[ \Pd\left( D'_t(v_r) = d, Z_r'=1 \middle | \mathcal{F}'_{s+1}\right) \middle | \mathcal{F}_{s+1}\right]\\
	& = \Ed \left[ \Pd\left( D'_t(v_r)= d, Z'_r=1 \; \middle | \; D_{s+1}'(v_r)\right) \middle | \mathcal{F}_{s+1}\right].
	\end{split}
	\end{equation}
	Now, observe that if $D_{s+1}(v_r) = D'_{s+1}(v_r)$, then 
	\[
	\Pd\left( D_t(v_r) = d, Z_r=1 \; \middle | \; D_{s+1}(v_r)\right) = \Pd\left( D'_t(v_r) = d, Z'_r=1 \; \middle | \; D_{s+1}'(v_r)\right),
	\]
	since both processes evolve with the same distribution. Furthermore, at time $s$, we have that ~$D_{s}(v_r) = D'_{s}(v_r)$, for all $r\le s$, thus, the number of vertices which have $D_{s+1} \neq D_{s+1}'$ is at most $4$. By the definition of $M_s$ and the above observations, the increment $|\Delta M_s|$ is equal to the sum below
	\begin{equation}\label{ineq:Ms}
	\begin{split}
	\left |\sum_{i=1}^t \Ed \left[ \Pd\left( D_t(v_r) = d,Z_r=1 \; \middle | \; D_{s+1}(v_r)\right) - \Pd\left( D'_t(v_r) = d,Z_r=1 \; \middle | \; D_{s+1}'(v_r)\right) \; \middle | \; \mathcal{F}_{s+1}\right] \right |
	\end{split}
	\end{equation}
	and all we have concluded so far leads to the following bound from above
	\begin{equation}\label{ineq:sMs}
	\begin{split}
	| M_{s+1} - M_s| & \le 
	\Ed \left[ \sum_{r=1}^t \mathbb{1}\{D_{s+1}(i)\neq D'_{s+1}(i) \}\; \middle | \; \mathcal{F}_{s+1}\right] \le 4,
	\end{split}
	\end{equation}
	which concludes the proof.
\end{proof}
The next step is to bound the conditional quadratic variation of $\{M_s\}_{s\ge 1}$ in order to apply Freedman's inequality, which is done in the lemma below.
\begin{lemma}[Upper bound for the quadratic variation]\label{lemma:boundquad} Let $f$ be any edge-step function and $\{M_s\}_{s\ge 1}$ be as in (\ref{def:M}). Then, the following bound holds
	\begin{equation}
	\Ed\left[ \left(M_{s+1} -  M_s\right)^2 \middle | \mathcal{F}_s\right] \le \frac{10d^2N_s( \le d,f)}{s},
	\end{equation}
	for all time $s$ and degree $d$.
\end{lemma}
\begin{proof} By (\ref{ineq:Ms}) and Jensen's inequality we have that $(M_{s+1} - M_s)^2$ is bounded from above by
	\begin{equation}
	\begin{split}
	\Ed \left[ \left(\sum_{r=1}^t \Pd\left( D_t(v_r) = d,Z_r=1 \; \middle | \; \mathcal{F}_{s+1}\right) - \Pd\left( D'_t(v_r) = d, Z'_r=1 \; \middle | \; \mathcal{F}'_{s+1}\right)\right)^2 \; \middle | \; \mathcal{F}_{s+1}\right]. 
	\end{split}
	\end{equation}
	Taking the conditional expectation \textit{w.r.t} $\mathcal{F}_s$, using the tower property and recalling that we must have $D_s(v_r) \le d$ yields
	\begin{equation}\label{ineq:quadratic1}
	\begin{split}
	\Ed\left[ \left(\Delta M_s\right)^2 \middle | \mathcal{F}_s\right] 
	& \le  \Ed \left[ \left(\sum_{r=1}^t \mathbb{1}\{D_{s+1}(v_r) \neq D'_{s+1}(v_r) \} \mathbb{1}\{D_s(v_r) \le d\}\right)^2 \; \middle | \; \mathcal{F}_{s}\right].
	\end{split}
	\end{equation}
	
	Now, observe that the following upper bound holds deterministically
	\begin{equation}
	\mathbb{1}\{D_{s+1}(v_r) \neq D'_{s+1}(v_r) \} \le \Delta D_s(v_r) + \Delta D'_s(v_r)
	\end{equation}
	and identities (\ref{eq:vardeg1}) and (\ref{eq:vardeg2}) give us
	\begin{equation}
	\Ed \left[ \Delta D'_s(v_r) \middle | \mathcal{F}_s\right] = \Ed \left[ \Delta D_s(v_r) \middle | \mathcal{F}_s\right] = \left(1-\frac{f(s+1)}{2}\right)\frac{D_s(v_r)}{s}, 
	\end{equation}
	which, in turn, leads to
	\begin{equation}\label{ineq:diff1}
	\Pd\left( D_{s+1}(v_r) \neq D'_{s+1}(v_r) \middle | \mathcal{F}_s\right) \le \frac{2D_s(v_r)}{s},
	\end{equation}
	for all $r \in \{1,\cdots, t\}$. For~$u\geq 1$, using that the product $\Delta D_s(v_r)\Delta D_s(v_u)$ is non-zero if and only if both vertices are selected at the same time, and that $\Delta D_s(v_r)$ and~$\Delta D'_s(v_u)$ are independent given~$\mathcal{F}_s$, we also derive
	\begin{equation}\label{ineq:diff2}
	\Pd\left( D_{s+1}(v_r) \neq D'_{s+1}(v_r), D_{s+1}(v_u) \neq D'_{s+1}(v_u) \middle | \mathcal{F}_s\right) \le \frac{4D_s(v_r)D_s(v_u)}{s^2}.
	\end{equation}
	Expanding the summand on the RHS of (\ref{ineq:quadratic1}) and substituting (\ref{ineq:diff1}) and (\ref{ineq:diff2}) in it, we obtain
	\begin{equation}
	\begin{split}
	\lefteqn{\Ed\left[ \left(\Delta M_s\right)^2 \middle | \mathcal{F}_s\right]}\phantom{**}\\ & \le 2\sum_{r=1}^t\frac{D_s(v_r)\mathbb{1}\{D_s(v_r) \le d\}}{s} + 8\sum_{1\le r < u\le t}\frac{D_s(v_r)D_s(v_u)\mathbb{1}\{D_s(v_r) \le d, D_s(v_u) \le d\}}{s^2} \\
	& \le \frac{2dN_s( \le d, f)}{s} + 8d^2\sum_{1\le r <u\le t}\frac{\mathbb{1}\{D_s(v_r) \le d, D_s(v_u) \le d\}}{s^2} \\
	&\le \frac{10d^2N_s( \le d, f)}{s},
	\end{split}
	\end{equation}
	since
	\begin{equation*}
	\begin{split}
	\lefteqn{\sum_{1\le r < u\le t}\mathbb{1}\{D_s(v_r) \le d, D_s(v_u) \le d\}}\phantom{********}\\  &\le \left( \sum_{r=1}^t\mathbb{1}\{D_s(v_r) \le d\}\right)\left( \sum_{u=1}^t\mathbb{1}\{D_s(v_u) \le d\}\right)
	 = N^2_s(\le d, f)
	\end{split}
	\end{equation*}
	and $N_s( d,f)$ is less than $s$ deterministically. This finishes the proof.
\end{proof}
Now we are able to prove a general concentration result for $N_t(d,f)$, which holds for any edge-step function $f$. Then, we obtain Theorem~\ref{thm:powerlaw} as a consequence of exploiting additional information about $f$.
\subsection{The General case} For the general picture, our estimates of the deviation of $N_t(d, f)$ from its expected value depend on
\[
\sum_{s=1}^t\frac{1}{s}\sum_{r=1}^sf(r)
\]
which cannot be well estimated in this degree of generality. In this section we will prove a general concentration result, which holds for any $f$, but later we will see that this result can be very sharp if more information on the asymptotic behavior of $f$ is provided. For now, our goal is to prove the proposition below
\begin{proposition}\label{thm:conce}Let $f$ be any edge-step function. Then, for all $\lambda >0$ and $d \in \N$ it follows that 
	\begin{equation}
	\P \left( \left| N_t( d,f) - \Ed\left[N_t( d,f)\right]\right| \ge \lambda \right) \le \exp \left \lbrace -\frac{\lambda^2}{2\sigma^2_{d,t} + 8\lambda/3}\right \rbrace + \exp\left \lbrace -\frac{\lambda^2}{2F(t) +4\lambda/3} \right \rbrace,
	\end{equation}
	where
	\begin{equation}
	\label{eq:concesigma}
	\sigma^2_{d,t} := 10d^2\sum_{s=1}^{t -1}  \frac{F(s)+\lambda}{s}.
	\end{equation}
\end{proposition}
\begin{proof} We apply Freedman's inequality (Theorem~\ref{teo:freedman}) to the Doob martingale~$\{M_s\}_{s \ge 1}$ defined on (\ref{def:M}). Before, however, it will be important to control the number of vertices at time~$t$,~$V_t$. Recall that~$V_t$ is~$1$ plus the sum of the independent random variables~$Z_2, \cdots, Z_t$, and that~$Z_s \stackrel{d}{=} \mathrm{Ber}(f(s))$. Thus $V_t - F(t)$ is a mean zero martingale whose increments are bounded by $2$. And since the $Z_s$'s are independent, it follows that
	\begin{equation}
	\begin{split}
	\sum_{s=1}^{t-1} \Ed \left[\left( V_{s+1}-F(s+1) - V_{s} + F(s)\right)^2\; \middle | \; \mathcal{F}_s\right] & = \sum_{s=1}^{t-1} \Ed \left[\left( Z_{s+1} - f(s+1)\right)^2\; \middle | \; \mathcal{F}_s\right]  \le  F(t). \\
	\end{split}
	\end{equation}
	Then, applying Freedman's inequality on the martingale $V_t - F(t)$, with $\sigma^2 = F(t)$, we obtain that
	\begin{equation}\label{ineq:stop}
	\Pd\left( \max_{s \le t} \{V_s -F(s) \}\ge \lambda \right) \le \exp\left \lbrace -\frac{\lambda^2}{2F(t) +4\lambda/3} \right \rbrace.
	\end{equation}
	Now, for a fixed $\lambda > 0$, define the stopping time
	\begin{equation}
	\tau := \inf \left \lbrace s \ge 1 \; \middle | \; V_s - F(s) \ge \lambda\right \rbrace.
	\end{equation}
	Observe that (\ref{ineq:stop}) gives us
	\begin{equation}
	\Pd\left( \tau \le t \right) = \Pd\left( \max_{s \le t} \{V_s -F(s) \}\ge \lambda \right) \le \exp\left \lbrace -\frac{\lambda^2}{2F(t) +4\lambda/3} \right \rbrace.
	\end{equation}
	Now consider the stopped martingale $\{M_{s\wedge\tau}\}_{s \ge 1}$, whose conditional quadratic variation is bounded in the following way
	\begin{equation}
	\begin{split}
	\sum_{s=1}^{t-1} \Ed\left[ \left(\Delta M_{s\wedge \tau}\right)^2 \middle | \mathcal{F}_s\right]
	&\stackrel{\text{Lemma \ref{lemma:boundquad}}}{\le}\, \sum_{s=1}^{t-1} \frac{10d^2N_s(\le d, f)\mathbb{1}\{s \leq \tau\}}{s} \\
	& \le 10d^2\sum_{s=1}^{t\wedge \tau -1}\frac{V_s}{s}  \le  10d^2\sum_{s=1}^{t\wedge \tau -1}  \frac{F(s)+\lambda}{s},
	\end{split}
	\end{equation}
	deterministically, since the number of vertices having degree at most $d$ is less than the total number of vertices, and $V_s \le F(s) +\lambda$ whenever $s < \tau$. Recalling~\eqref{eq:concesigma}, we have that the LHS above is smaller than or equal to~$\sigma^2_{d,t}$. Defining then
	\[
	W_t := \sum_{k=1}^{t-1} \mathbb{E} \left[(M_{k+1}-M_k)^2\middle|\mathcal{F}_k \right],
	\]
	we have by Freedman's inequality, 
	\begin{equation}
	\Pd \left( \left | M_{t \wedge \tau } - \Ed N_{t \wedge \tau}( d,f) \right| \ge \lambda, W_{t\wedge \tau } \le \sigma^2_{d,t}\right) \le \exp \left \lbrace -\frac{\lambda^2}{2\sigma^2_{d,t} + 8\lambda/3}\right \rbrace.
	\end{equation}
	Finally, we obtain
	\begin{equation}
	\begin{split}
	\Pd \left( \left | M_{t} - \Ed N_{t}( d,f) \right| \ge \lambda\right) & \le \Pd \left( \left | M_{t \wedge \tau } - \Ed N_{t \wedge \tau}( d,f) \right| \ge \lambda, \tau > t\right) +  \Pd \left( \tau \le t \right)\\ 
	& \le \exp \left \lbrace -\frac{\lambda^2}{2\sigma^2_{d,t} + 8\lambda/3}\right \rbrace + \exp\left \lbrace -\frac{\lambda^2}{2F(t) +4\lambda/3} \right \rbrace,
	\end{split}
	\end{equation}
	finishing the proof.
\end{proof}
\subsection{Index of Regular variation in $(-1,0]$} Now, we will explore Proposition~\ref{thm:conce} when more properties of~$f$ are available in order to prove Theorem~\ref{thm:powerlaw}. As we will see, information about the asymptotic behavior of ~$f$ is enough to derive useful concentration results. Our goal is to prove that the fluctuations around the mean of $N_t(d,f)$ are of order $\sqrt{F(t)}$, which can be of order much smaller than~$\sqrt{t}$, as discussed in the beginning of this section. 
\begin{proof}[Proof of Theorem~\ref{thm:powerlaw}]
	We apply Proposition \ref{thm:conce} combined with the fact that we are now considering edge-step functions which are regularly varying, which gives us extra knowledge about the quantities involved in the statement of Proposition~\ref{thm:conce}.
	
	We begin observing that by Lemma~\ref{lemma:rvint} we have
	\begin{equation}\label{eq:integral}
	F(t) \sim \int_1^tf(s)ds \sim (1-\gamma)^{-1}\ell(t)t^{1-\gamma},
	\end{equation}
	for $\gamma \in [0,1)$. Consequently, we have that
	\begin{equation}
	\sum_{s=1}^t\frac{F(s)+\lambda}{s} \le (1-\gamma)^{-1} F(t)+\lambda\log(t).
	\end{equation}
	We set $\lambda = A\sqrt{40d^2(1-\gamma)^{-1}F(t)}$ with~$A<\sqrt{F(t)(1-\gamma)^{-1}}(4d\log(t))^{-1}$. Using Proposition~\ref{thm:conce} we obtain that for large enough $t$
	\begin{equation}
	\label{ineq:estifinal}
	\begin{split}
	\lefteqn{\P \left( \left| N_t( d,f) - \Ed\left[N_t( d,f)\right]\right| \ge A\sqrt{40d^2F(t)(1-\gamma)^{-1}} \right) }
	\phantom{**}
	\\
	&\le \exp \left \lbrace -\frac{\lambda^2}{20 d^2(1-\gamma)^{-1} F(t)+\lambda(20d^2\log(t)+ 8/3)}\right \rbrace + \exp\left \lbrace -\frac{\lambda^2}{2F(t) +4\lambda/3} \right \rbrace
	\\
	&\le \exp \left \lbrace -\frac{A^2\cdot 40d^2 (1-\gamma)^{-1}F(t)}{20d^2 (1-\gamma)^{-1} F(t)+A\sqrt{40d^2 (1-\gamma)^{-1}F(t)}\cdot(20d^2\log(t)+ 8/3)}\right \rbrace 
	\\
	&\quad+ \exp\left \lbrace -\frac{A^2\cdot 40d^2 (1-\gamma)^{-1}F(t)}{2F(t) +4/3\cdot A\sqrt{40d^2 (1-\gamma)^{-1}F(t)}} \right \rbrace
	\\
	& \leq 2\exp \left \lbrace -A^2\right \rbrace.
	\end{split}
	\end{equation}
	To prove the Theorem from the above result, note that by triangle inequality and the fact that~$N_t( d,f)\leq V_t(f)$ deterministically, we may obtain
	\begin{equation}\label{ineq:tri}
	\begin{split}
	\left | \hat{P}_t( d) - \frac{\Ed N_t( d,f)}{F(t)}\right| 
	& \le \left| \frac{N_t(d,f)(F(t)-V_t(f))}{V_t(f)F(t)}\right | + \left| \frac{N_t( d,f)}{F(t)} - \frac{\Ed N_t( d,f)}{F(t)}\right |
	\\
	& \le \left| \frac{V_t(f)}{F(t)} - 1\right | + \left| \frac{N_t( d,f)}{F(t)} - \frac{\Ed N_t( d,f)}{F(t)}\right |
	\end{split}
	\end{equation}
	By the multiplicative form of the Chernoff bound, we have
	\begin{equation}
	\Pd \left( \left| \frac{V_t}{F(t)} - 1\right | > \frac{A}{\sqrt{F(t)}}\right) \le \exp\left\lbrace -\frac{A^2}{3} \right\rbrace.
	\end{equation}
	The second term is then bounded by~(\ref{ineq:estifinal}), giving
	\begin{equation}
	\Pd \left(  \left| \hat{P}_t( d) - \frac{\Ed N_t( d,f)}{F(t)}\right| > 10d\frac{A}{\sqrt{(1-\gamma)F(t)}} \right) \leq  \exp\left\lbrace -\frac{A^2}{3} \right\rbrace+ 2e^{-A^2},
	\end{equation}
	finishing the proof of the Theorem.
\end{proof}

%% file: roteiro-conv-qc.tex
In this section we prove Theorem~\ref{t:notpowerlaw}, which states that when the index of regular variation is less than $-1$ the empirical distribution $\{\hat{P}_t(d,f)\}_{t \in \N}$ converges to zero almost surely for any fixed $d$. The mass on finite degrees is completely lost in this regime. We start by showing that this phenomenon happens in expectation.
\begin{proposition}\label{prop:l1conv}Let $f \in \mathrm{RES}(-\gamma)$, with $\gamma \in [1,\infty)$. Then, for all $d\in \N$, we have that 
	\[
	\lim_{t \rightarrow \infty} \frac{\Ed N_t(d,f)}{F(t)} = 0.
	\]
\end{proposition}
\begin{proof}We proceed by induction on $d$. Again, by the Representation Theorem (Theorem~\ref{thm:repthm}), there exists a slowly varying function~$\ell$ such that $f(t)=t^{-\g}\ell(t)$ for all $t \ge 1$. In order to simplify our writing, we let $a_t(d)$ be $\Ed N_{t}(d,f)$.
	
	\underline{Base case of the induction.} According to Lemma \ref{lemma:recurntd}, we have, for $d=1$
	\begin{equation}
	a_{t+1}(1) \leq \left(1 - \frac{1}{t} + \frac{\ell(t)}{2t^{1+\g}}\right)a_t(1) + \frac{\ell(t+1)}{(t+1)^{\gamma}} + O\left( \frac{F(t)}{t^{2}}\right).
	\end{equation}
	Expanding the above recurrence relation yields
	\begin{equation}\label{ineq:recu}
	\begin{split}
	a_{t+1}(1) &\leq \frac{\ell(t+1)}{(t+1)^{\gamma}}+O\left( \frac{F(t)}{t^{2}}\right) + \sum_{s=1}^t \left[\left(\frac{\ell(s)}{s^{\g}}+O\left( \frac{F(s)}{s^{2}}\right)\right) \prod_{r=s}^t\left(1 - \frac{1}{r} + \frac{\ell(r)}{2r^{1+\g}}\right)  \right]\\
	& \le \exp\left\lbrace \sum_{r=1}^{\infty} \frac{\ell(r)}{2r^{1+\g}} \right\rbrace\sum_{s=1}^t \left[\left(\frac{\ell(s)}{s^{\g}}+O\left( \frac{F(s)}{s^2}\right)\right) \exp\left\lbrace -\sum_{r=s}^{t} \frac{1}{r} \right\rbrace\right] +o(1)\\
	& \le  \exp\left\lbrace \sum_{r=1}^{\infty} \frac{\ell(r)}{2r^{1+\g}} \right\rbrace \frac{1}{t}\sum_{s=1}^t \left[\left(\frac{\ell(s)}{s^{\g-1}}+O\left( \frac{F(s)}{s}\right)\right) \right] +o(1).
	\end{split}
	\end{equation}
	When $\g > 1$ it is straightforward to verify that $a_t(d) < C_d$ for all $t \ge 1$. Thus, from now on, we assume $\gamma =1$, which is the hardest case. Observe that, by Karamata's Theorem (Theorem~\ref{thm:karamata}), it follows that
	\begin{equation}\label{ineq:obs1}
	\exp\left\lbrace \sum_{r=1}^{\infty} \frac{\ell(r)}{2r^{1+\g}} \right\rbrace \le c_1
	\end{equation}
	and by Corollary~\ref{cor:a3}, we also have
	\begin{equation}\label{ineq:obs2}
	\lim_{s \rightarrow \infty} \frac{F(s)}{s} = 0 \implies \lim_{t \rightarrow \infty}\frac{1}{t}\sum_{s=1}^{t}\frac{F(s)}{s} =0.
	\end{equation}
	By Lemma~\ref{lemma:rvint}, for large enough $t$, we have that
	\begin{equation}
	\begin{split}
		\frac{1}{t}\sum_{s=1}^t\ell(s) \le \frac{2t\ell(t)}{t} = 2\ell(t).
	\end{split}
	\end{equation}
	Therefore, we have that, for some positive constant $C$,
	\begin{equation}
	a_t(1) \le C\ell(t)
	\end{equation}
	and by Corollary~\ref{cor:lFt} (whose proof we postpone to the Appendix) it follows that
	\begin{equation}
	\lim_{t \rightarrow \infty}\frac{a_t(1)}{F(t)} = 0.
	\end{equation}
	concluding the base step.
	
	\underline{Inductive step.} Assume that for all $k\le d-1$ there exists $C_k$ such that
	\begin{equation}
	\label{eq:gama1exp}
	a_t(k) \le C_k \ell(t).
	\end{equation}
	Recall the recurrence relation given by (\ref{eq:end}), which gives us
	\[
	a_{t+1}(d) \leq\left(1 - \frac{d}{t} + \frac{d\ell(t)}{2t^{1+\g}}\right)a_t(d) + \left( \frac{d-1}{t} - \frac{(d-1)\ell(t)}{2t^{1+\g}}\right)a_{t}(d-1) + O_d\left( \frac{F(t)}{t^2}\right).
	\]
	Expanding the above equality and recalling that $\gamma = 1$, we obtain
	\begin{equation}
	\begin{split}
	a_{t+1}(d) &= \sum_{s=1}^t \left[\left(\left(\frac{d-1}{s} - \frac{(d-1)\ell(s)}{2s^{2}}\right)a_{s}(d-1)+O_d\left( \frac{F(s)}{s^2}\right) \right)\prod_{r=s+1}^t\left(1 - \frac{d}{r} + \frac{d\ell(r)}{2r^{2}}\right)\right] \\
	&\le \frac{c_d}{t^d}\sum_{s=1}^t \left[ s^{d-1}a_s(d-1) +O_d\left( F(s)s^{d-2}\right)\right].
	\end{split}
	\end{equation}
	From Corollary~\ref{cor:a3} it follows that, for some~$\varepsilon>0$,
	\[
	\lim_{t \to \infty} \frac{c_d}{t^d}\sum_{s=1}^t O_d\left( \frac{F(s)}{s}s^{d-1}\right) \leq \frac{c_d}{t^d}\sum_{s=1}^t O_d\left( \frac{s^\varepsilon}{s}s^{d-1}\right)\leq c_d t^{-(1-\varepsilon)}
	\]
	Finally, the inductive hypothesis  and Karamata's theorem lead to
	\[
	\frac{1}{t^d}\sum_{s=1}^ts^{d-1}a_s(d-1) \le \frac{C_{d-1}}{t^d}\sum_{s=1}^ts^{d-1}\ell(s) \le c_d'\ell(t),
	\]
	proving the inductive step, since $\ell(t) \geq t^{-(1-\varepsilon)}$ for sufficiently large~$t$.
	
	Combining~\eqref{eq:gama1exp} with Corollary~\ref{cor:lFt} it is proved that
	\[
	\lim_{t \rightarrow \infty} \frac{a_t(d)}{F(t)} = 0	
	\]
	for all $d \in \N$, finishing the proof.
\end{proof}
From Proposition~\ref{prop:l1conv} we will prove the \textit{a.s.} convergence employing a second moment estimate. For this we will need a new definition and a few lemmas.
\begin{definition}[$d$-admissible vectors] Given $d,t,r,s\in\mathbb{N}$, with $r<s<t$ and two vertices $v_s$ and $v_r$ born at time~$r$ and~$s$ respectively, we say that two vectors $\vec{x}_{s,t}:=(x_u)_{u=s+1}^t$ and $\vec{y}_{r,t}:=(y_u)_{u=r+1}^t$ are \textbf{$d$-admissible} for $v_s$ and $v_r$ if~$x_u, y_u \in \{0,1,2\}$ for all $u$, the sum of their coordinates is at most $d$, $y_s \neq 2$ and the vectors do not have a $2$ in the same coordinate.
\end{definition}
Observe that given a vertex $v_s$, the vector $\vec{x}_{s,t} \in \{0,1,2\}^{t-s}$ induces an event in which the trajectory of the degree of $v_s$ up to time $t$ is completely characterized by said vector. More specifically, $\vec{x}_{s,t}=(x_u)_{u=s+1}^t$ characterizes the event
\[
\{\Delta D_t(v_s) = x_t\} \cap \cdots \cap \{\Delta D_{s+1}(v_s) = x_{s+1}\} \cap \{Z_s = 1\}.
\]
Thus, two vectors are $d$-admissible if the events induced by them imply that both~$D_t(v_r)$ and~$D_t(v_s)$ are at most $d$ and that their intersection is not empty. Moreover, given two $d$-admissible vectors~$\vec{x}_{s,t}$ and $\vec{y}_{r,t}$ we denote by $\mathbb{P}_{\vec{x}_{s,t},\vec{y}_{r,t}}$ the distribution $\Pd$ conditioned on the intersection of the events induced by the vectors. Also, to simplify our writing, fixed the vertices $v_s$ and $v_r$ and two $d$-admissible vectors, we write for all $u$
\begin{equation}
\Delta_u := \Delta D_u(v_s); \;\; \Delta'_u := \Delta D_u(v_r).
\end{equation}
The following Lemma is the first step in obtaining a decorrelation estimate that will allow us to estimate the variance of~$N_t(\leq d, f)$, the number of vertices at time~$t$ with degree lesser than or equal to~$d$.
\begin{lemma}\label{lemma:decor1} Let $\vec{x}_{s,t+1}=(x_u)_{u=s+1}^{t+1}$ and $\vec{y}_{r,t+1}=(y_u)_{u=r+1}^{t+1}$ be two $d$-admissible vectors for some $d \in \N$ and vertices $v_s$ and $v_r$. Then,
	\begin{equation}
	\label{eq:dec1}
		\begin{split}
				\lefteqn{\mathbb{P}_{\vec{x}_{s,t},\vec{y}_{r,t}}\left(\Delta_t = x_{t+1}, \Delta'_t = y_{t+1} \right) }\phantom{*******}\\
				&\leq \left(1+O\left(\frac{\ell(t)+d}{t}\right)\right)\mathbb{P}_{\vec{x}_{s,t},\vec{y}_{r,t}}\left(\Delta_t = x_{t+1} \right)\mathbb{P}_{\vec{x}_{s,t},\vec{y}_{r,t}}\left(\Delta'_t = y_{t+1} \right),
		\end{split}
	\end{equation}
for all $t > s$. Furthermore, for the special case where~$x_{t+1}=y_{t+1}=0$, we have, also for all~$t > s$,
\begin{equation}
\label{eq:dec2}
\mathbb{P}_{\vec{x}_{s,t},\vec{y}_{r,t}}\left(\Delta_t =0, \Delta'_t = 0 \right)\leq
\mathbb{P}_{\vec{x}_{s,t},\vec{y}_{r,t}}\left(\Delta_t = 0 \right)\mathbb{P}_{\vec{x}_{s,t},\vec{y}_{r,t}}\left(\Delta'_t = 0 \right).
\end{equation}	
\end{lemma}
\begin{proof} The proof is done by direct computation. We compute the probabilities of all possible combinations for $x_{t+1}$ and $y_{t+1}$ in $\{0,1,2\}$ and compare them. We will write the degree~$d_t(v_s)$ in lower case meaning the degree of $v_s$ at time $t$ according to the event induced by the vector~$\vec{x}_{s,t}$, analogously defining~$d_t(v_r)$ for~$v_r$. Note that since the two vectors are $d$-admissible,~$d_t(v_s)$ and~$d_t(v_r)$ are both less than $d$. We have
\begin{equation}\label{eq:d0}
\begin{split}
\mathbb{P}_{\vec{x}_{s,t},\vec{y}_{r,t}}\left(\Delta_t = 0 \right) & =\left(1-\frac{d_t(v_s)}{2t}\right)\left[1-(1-f(t+1))\frac{d_t(v_s)}{2t}\right]
\end{split}
\end{equation}
\begin{equation}\label{eq:d1}
\begin{split}
\mathbb{P}_{\vec{x}_{s,t},\vec{y}_{r,t}}\left(\Delta_t = 1 \right) & = \frac{d_t(v_s)}{2t}\left(f(t+1)+2(1-f(t+1))-2(1-f(t+1))\frac{d_t(v_s)}{2t}\right) \\
&= \frac{d_t(v_s)}{t}\left[1+ O(f(t+1)+dt^{-1})\right]
\end{split}
\end{equation}
And finally
\begin{equation}\label{eq:d2}
\begin{split}
\mathbb{P}_{\vec{x}_{s,t},\vec{y}_{r,t}}\left(\Delta_t = 2 \right) & = (1-f(t+1))\frac{d_t^2(v_s)}{4t^2}.
\end{split}
\end{equation}
Now, we consider the cases in which $\Delta_{u}$ and $\Delta'_{u}$ change simultaneously.
\begin{equation}\label{eq:d00}
\begin{split}
\mathbb{P}_{\vec{x}_{s,t},\vec{y}_{r,t}}\left(\Delta_t = 0, \Delta'_t = 0 \right)
& = \left(1-\frac{d_t(v_s)}{2t}-\frac{d_t(v_r)}{2t}\right)\left[1 - (1-f(t+1))\left(\frac{d_t(v_s)}{2t}+\frac{d_t(v_r)}{2t}\right)\right]
\end{split}
\end{equation}
\begin{equation}\label{eq:d10}
\begin{split}
\lefteqn{\mathbb{P}_{\vec{x}_{s,t},\vec{y}_{r,t}}\left(\Delta_t = 1, \Delta'_t = 0 \right) }\phantom{******}
\\&=
\frac{d_t(v_s)}{2t}f(t+1)+2(1-f(t+1))\frac{d_t(v_s)}{2t}\left(1-\frac{d_t(v_s)}{2t}-\frac{d_t(v_r)}{2t}\right)
\\
&=\frac{d_t(v_s)}{t}(1+O(f(t+1)+dt^{-1})).
\end{split}
\end{equation}
\begin{equation} \label{eq:d11}
\begin{split}
\mathbb{P}_{\vec{x}_{s,t},\vec{y}_{r,t}}\left(\Delta_t = 1, \Delta'_t = 1 \right) & = 2(1-f(t+1))\frac{d_t(v_s)d_t(v_r)}{4t^2}
\end{split}
\end{equation}
Finally
\begin{equation}\label{eq:d20}
\begin{split}
\mathbb{P}_{\vec{x}_{s,t},\vec{y}_{r,t}}\left(\Delta_t = 2, \Delta'_t = 0 \right) & = (1-f(t+1))\frac{d^2_t(v_s)}{4t^2}.
\end{split}
\end{equation}
These cases are enough to cover all possible combinations. The result then follows by a direct comparison between product of probabilities given by (\ref{eq:d0},\ref{eq:d1}, \ref{eq:d2}) with those obtained in (\ref{eq:d00},\ref{eq:d10}, \ref{eq:d11}, \ref{eq:d20}). In particular, we note that from~\eqref{eq:d0} and~\eqref{eq:d00} we obtain
\begin{equation}
\label{eq:d00b}
\begin{split}
\lefteqn{\mathbb{P}_{\vec{x}_{s,t},\vec{y}_{r,t}}\left(\Delta_t = 0 \right)\mathbb{P}_{\vec{x}_{s,t},\vec{y}_{r,t}}\left(\Delta'_t = 0 \right)}\phantom{******}
\\&=
\left(1-\frac{d_t(v_s)}{2t}-\frac{d_t(v_r)}{2t}+\frac{d_t(v_s)d_t(v_r)}{4t^2}\right)
\\
&\quad\times\left[1 - (1-f(t+1))\left(\frac{d_t(v_s)}{2t}+\frac{d_t(v_r)}{2t} \right)+(1-f(t+1))^2 \frac{d_t(v_s)d_t(v_r)}{4t^2} \right]
\\
&\geq 
\mathbb{P}_{\vec{x}_{s,t},\vec{y}_{r,t}}\left(\Delta_t = 0 ,\Delta'_t = 0 \right),
\end{split}
\end{equation}
finishing the proof of the lemma.

\end{proof}
For a fixed vertex $v_s$ and $d \in \N$ define the event
\begin{equation}
	E_{t,d}(v_s) := \left \lbrace D_t(v_s) \le d, Z_s = 1 \right \rbrace.
\end{equation}
In the next lemma we prove that the events $E_{t,d}(v_r)$ and $E_{t,d}(v_s)$ are almost uncorrelated.
\begin{lemma}\label{lemma:decor2} For~$d,r,s,t \in \N$ with~$r<s\leq t$ we have
\begin{equation}
\begin{split}
\Pd\left(E_{t,d}(v_s), E_{t,d}(v_r)\right) & \le  \left(1+ dO\left(\frac{\ell(s)+d}{s}\right)\right)\Pd\left(E_{t,d}(v_s)\right)\Pd\left( E_{t,d}(v_r)\right) .
\end{split}
\end{equation}
\end{lemma}
\begin{proof} Fix two $d$-admissible vectors $\vec{x}_{s,t}=(x_u)_{u=s+1}^t$ and $\vec{y}_{r,t}=(y_u)_{u=r+1}^t$ and denote by~$\Xi(\vec{x}_{s,t})$ the event
\begin{equation}
\Xi(\vec{x}_{s,t}):= \{\Delta_{t-1}= x_{t}\}\cap \cdots \{\Delta_s = x_{s+1}\} \cap \{Z_s = 1\},
\end{equation}
and analogously define~$\Xi(\vec{y}_{r,t})$. Observe that for each $m \in {s+1, \cdots, t}$ we have
\begin{equation}\label{eq:markov}
\mathbb{P}_{\vec{x}_{s,m},\vec{y}_{r,m}}\left(\Delta_m = x_{m+1}\right) = \Pd\left(\Delta_m = x_{m+1} \; \middle | \; \Delta_{m-1} = x_m, \cdots, \Delta_s = x_{s+1}, Z_s = 1\right),
\end{equation}
where~$\mathbb{P}_{\vec{x}_{s,m},\vec{y}_{r,m}}$ is defined for the restrictions of the vectors~$\vec{x}_{s,t}$ and~$\vec{y}_{r,t}$ up to time~$m$. We apply Lemma~\ref{lemma:decor1} iteratively and then use (\ref{eq:markov}) to regroup the terms in a convenient way. For the first step we note that
\begin{equation*}
\begin{split}
\lefteqn{\Pd\left(\Xi(\vec{x}_{s,t}), \Xi(\vec{y}_{r,t})\right)}\\
& \leq\left(1+ O\left(\frac{\ell(t)+d}{t}\right)\right)\mathbb{P}_{\vec{x}_{s,t-1},\vec{y}_{r,t-1}}(\Delta_{t-1} = x_{t})\mathbb{P}_{\vec{x}_{s,t-1},\vec{y}_{r,t-1}}(\Delta'_{t-1} = y_{t})\Pd\left(\Xi(\vec{x}_{s,t-1}), \Xi(\vec{y}_{r,t-1})\right)
\end{split}
\end{equation*}
We iterate this the procedure until $u=s+1$. The case $u=s$ we must handle in a slightly different way. Note that, when $u=s$ we have to deal with the term
\begin{equation}
\begin{split}
\mathbb{P}_{\vec{x}_{s,s-1},\vec{y}_{r,s-1}}\left(\Delta'_s = y_s, Z_s = 1\right) = \mathbb{P}_{\vec{x}_{s,s-1},\vec{y}_{r,s-1}}\left(\Delta'_s = y_s \middle | Z_s = 1\right)f(s),
\end{split}
\end{equation}
since $Z_s$ is independent of $\mathcal{F}_{s-1}$. Now, for $y_s = 1$ we have
\begin{equation}
\begin{split}
\frac{\Pd_{\vec{x}_{s,s-1},\vec{y}_{r,s-1}}\left(\Delta_s ' = 1 \middle | Z_s = 1\right)}{\Pd_{\vec{x}_{s,s-1},\vec{y}_{r,s-1}}\left(\Delta_s '= 1 \right)} & = \frac{\frac{d_{s-1}(v_r)}{2(s-1)}}{f(s)\frac{d_{s-1}(v_r)}{2(s-1)}+2(1-f(s))\frac{d_{s-1}(v_r)}{2(s-1)}\left(1-\frac{d_{s-1}(v_r)}{2(s-1)}\right)} \\
& = \frac{1}{2-f(s)-2(1-f(s))\frac{d_{s-1}(v_r)}{2(s-1)}}
\\
&=\frac{1}{2}(1+O(f(s)+ds^{-1})).
\end{split}
\end{equation}
And for~$y_s = 0$ we get
\begin{equation}
\begin{split}
\frac{\Pd_{\vec{x}_{s,s-1},\vec{y}_{r,s-1}}\left(\Delta_s ' = 0 \middle | Z_s = 1\right)}{\Pd_{\vec{x}_{s,s-1},\vec{y}_{r,s-1}}\left(\Delta_s '= 0 \right)} & = \frac{1-\frac{d_{s-1}(v_r)}{2(s-1)}}{f(s)\left(1-\frac{d_{s-1}(v_r)}{2(s-1)}\right)+(1-f(s))\left(1-\frac{d_{s-1}(v_r)}{2(s-1)}\right)^2} \\
&=(1+O(f(s)+ds^{-1})).
\end{split}
\end{equation}
Iterating the procedure we obtain
\begin{equation*}
\begin{split}
\lefteqn{\Pd\left(\Xi(\vec{x}_{s,t}), \Xi(\vec{y}_{r,t})\right) }
\\& \leq \prod_{u=s}^{t}\left(1+ O\left(\frac{\ell(u)+d}{u}\right)\right)\mathbb{P}_{\vec{x}_{s,u-1},\vec{y}_{r,u-1}}(\Delta_{u-1} = x_{u})\mathbb{P}_{\vec{x}_{s,u-1},\vec{y}_{r,u-1}}(\Delta'_{u-1} = y_{u})\Pd(\Xi(\vec{y}_{r,s-1})).
\end{split}
\end{equation*}
Note that, since~$\vec{x}_{s,t}$ and~$\vec{y}_{r,t}$ are~$d$-admissible, in all but at most~$2d$ steps the increments are both~$0$. Therefore, by~\eqref{eq:d00b} and the fact that~$f(u) = \ell(u)/u$ is nonincreasing, we can use~(\ref{eq:markov}) to regroup separately all terms involving $v_s$ and $v_r$ to obtain
\begin{equation}
\label{eq:unco}
\begin{split}
\Pd\left(\Xi(\vec{x}_{s,t}), \Xi(\vec{y}_{r,t})\right) &\leq \left(1+ O\left(\frac{\ell(s)+d}{s}\right)\right)^{2d}\Pd\left(\Xi(\vec{x}_{s,t})\right)\Pd\left( \Xi(\vec{y}_{r,t})\right)
\\&\leq
\left(1+ dO\left(\frac{\ell(s)+d}{s}\right)\right)\Pd\left(\Xi(\vec{x}_{s,t})\right)\Pd\left( \Xi(\vec{y}_{r,t})\right).
\end{split}
\end{equation}

We can then use the above equation to get
\begin{equation}
\begin{split}
\Pd\left(E_{t,d}(v_s), E_{t,d}(v_r)\right) &=
\sum_{\substack{ \vec{x}_{s,t},\vec{y}_{r,t} \\ d-\text{admissible}}}\Pd\left(\Xi(\vec{x}_{s,t}), \Xi(\vec{y}_{r,t})\right)
\\
&\leq
\left(1+ dO\left(\frac{\ell(s)+d}{s}\right)\right)\sum_{\substack{ \vec{x}_{s,t},\vec{y}_{r,t} \\ d-\text{admissible}}}\Pd\left(\Xi(\vec{x}_{s,t})\right)\Pd\left( \Xi(\vec{y}_{r,t})\right)
\\
&\leq
\left(1+ dO\left(\frac{\ell(s)+d}{s}\right)\right)\sum_{\vec{x}_{s,t},\vec{y}_{r,t}}\Pd\left(\Xi(\vec{x}_{s,t})\right)\Pd\left( \Xi(\vec{y}_{r,t})\right)
\\
&=
\left(1+ dO\left(\frac{\ell(s)+d}{s}\right)\right)\Pd\left(E_{t,d}(v_s)\right)\Pd\left( E_{t,d}(v_r)\right) ,
\end{split}
\end{equation}
finishing the proof of the lemma.
\end{proof}
The next lemma shows that it is hard for earlier vertices to have small degrees.
\begin{lemma}\label{lemma:degoldv} Let $\delta \in (0,1)$ and $d \in \N$. Then, for $ r \le t^{1-\delta}$ we have that
	\[
	\Pd\left( D_t(v_r) \le d \;\;\middle | \;\; Z_r=1\right) \le e^{d}t^{-\delta/4}.
	\]
\end{lemma}
\begin{proof} Let $\tilde{G}_r$ be a possible realization of the process~$(G_t)_{t\geq 1}$ at time~$r$ such that the vertex~$v_r$ belongs to $V(\tilde{G}_r)$, and let~$\Pd_{\tilde{G}_r}$ be the distribution $\Pd$ conditioned on the event where $G_r=\tilde{G}_r$. By the simple Markov property,~$\Pd_{\tilde{G}_r}$ has the same distribution as our model started from~$\tilde{G}_r$. Now from (\ref{eq:d0}) and (\ref{eq:d1}), we obtain, for any step $u\geq r$,
\begin{equation}
\begin{split}
\Pd_{\tilde{G}_r}\left(\Delta D_u(v_r) \ge 1 \middle | \mathcal{F}_u\right) & = (2-f(u+1))\frac{D_u(v_r)}{2u} -(1-f(u+1))\frac{D^2_u(v_r)}{4u^2} \ge \frac{D_u(v_r)}{2u},
\end{split}
\end{equation}
since $D_u(v_r) \le 2u$ deterministically. Using the fact that the degree is at least one, we obtain that~$D_t(v_r)$ dominates a sum of independent random variables $\{Y_u\}_{u=t^{1-\delta}}^t$ where $Y_u \stackrel{\tiny d}{=} \mathrm{Ber}(1/2u)$. Observe that, bounding the sum by the integral, we obtain the following lower bound for the expectation of the degree of~$v_r$ under~$\Pd_{\tilde{G}_r}$
\begin{equation}
\mu_t := \Ed_{\tilde{G}_r} \left[\sum_{u=t^{1-\delta}}^{t}Y_u \right] \ge \frac{\delta}{2}\log t.
\end{equation}
Consequently,
taking $\e$ as
\begin{equation}
\e = 1- \frac{d}{\mu_t}
\end{equation}
and applying the Chernoff bound leads to
\begin{equation}
\begin{split}
\Pd_{\tilde{G}_r}\left( D_t(v_r) \le d \right) & \le  \Pd\left(\sum_{u=t^{1-\delta}}^{t} Y_u \le \left(1-\e\right)\mu_t\right) \le \exp\left \lbrace -\frac{\e^2 \mu_t}{2}\right \rbrace  \le \frac{e^d}{t^{\delta/4}}.
\end{split}
\end{equation}
Integrating over all possible graphs~$\tilde{G}_r$ gives the desired result.
\end{proof}
We have now the ingredients needed in order to bound~$\mathrm{Var}\left(N_t(\le d,f)\right)$, which in turn will let us finish the argument using the Chebyshev inequality, the Borel-Cantelli lemma and an elementary subsequence argument.
\begin{lemma}\label{lemma:var} For any $d \in \N$ and $ f \in \mathrm{RES}(-\g)$, with $\g \in [1, \infty)$, we have
	\[
	\mathrm{Var}\left(N_t(\le d,f)\right) \leq \Ed N_t(\le d,f)(1+o(1))
	\]
\end{lemma}
\begin{proof}By definition, we may write $N_t(\le d,f)$ as
	\begin{equation}
	N_t(\le d,f) = \sum_{s \le t} \mathbb{1}\{D_t(v_s) \le d \}Z_s=\sum_{s \le t} \mathbb{1}\{  E_{t,d}(v_s)  \}.
	\end{equation}
To control $N_t^2(\le d,f)$ we split the sum over $s$ into two sets: the vertices added before $t^{1-\delta}$ and those added after, for some small $\delta$. By Lemma~\ref{lemma:decor2}, for $s>t^{1-\delta}$ and $r<s$,we have that
\begin{equation}
\begin{split}
\lefteqn{ \Pd\left(E_{t,d}(v_s),E_{t,d}(v_r)\right) - \Pd\left(E_{t,d}(v_s)\right)\Pd\left(E_{t,d}(v_r)\right)}\phantom{*********************} \\& \le \frac{C(\ell(t^{1-\delta})+d)}{t^{1-\delta}}\Pd\left(E_{t,d}(v_s)\right)\Pd\left( E_{t,d}(v_r)\right) .
\end{split}
\end{equation}
Thus, 
\begin{equation}
\label{eq:var1}
\begin{split}
\lefteqn{\sum_{s=t^{1-\delta}}^t\sum_{r=1}^{s-1}\Pd\left(E_{t,d}(v_s),E_{t,d}(v_r)\right) - \Pd\left(E_{t,d}(v_s)\right)\Pd\left(E_{t,d}(v_r)\right)}\phantom{*********************}\\ & \le  C\frac{\Ed\left[N_t(\le d,f)\right]^2 (\ell(t^{1-\delta})+d)}{t^{1-\delta}}
\\
&\stackrel{ Lemma~\ref{prop:l1conv}}{\leq}
C\frac{\Ed\left[N_t(\le d,f)\right] (\ell(t^{1-\delta})+d)}{t^{1-\delta}}
\\
&=o(\Ed\left[N_t(\le d,f)\right]).
\end{split}
\end{equation}
Using Lemmas~\ref{lemma:degoldv} and~\ref{lemma:decor2}, and assuming $r<s\leq t^{1-\delta}$ we get
\begin{equation}
\begin{split}
 \Pd\left(E_{t,d}(v_r), E_{t,d}(v_s)\right)
&\leq C\Pd\left(E_{t,d}(v_r)\right) \Pd\left(E_{t,d}(v_s)\right)
\\
& \le C\Pd\left(E_{t,d}(v_s)\right)\Pd\left(D_t(v_r)\le d \middle |Z_r =1\right)f(r) 
\\
& \le C\Pd\left(E_{t,d}(v_s)\right)t^{-\delta/4}f(r).
\end{split}
\end{equation}
We therefore have
\begin{equation}
\sum_{s=1}^{t^{1-\delta}}\sum_{r=1}^{s-1} \Pd\left(E_{t,d}(v_r), E_{t,d}(v_s)\right) \le C\Ed\left[N_t(\le d,f)\right]\frac{F(t^{1-\delta})}{t^{\delta/4}} = o(\Ed\left[N_t(\le d,f)\right]),
\end{equation}
which implies, together with~\eqref{eq:var1},
\begin{equation}
\begin{split}
\mathrm{Var}\left(N_t(\le d,f)\right)&=\sum_{s=1}^t \Pd\left( E_{t,d}(v_s)\right)+2\sum_{s=1}^t\sum_{r=1}^{s-1}\Pd\left(E_{t,d}(v_r), E_{t,d}(v_s)\right)
\\
&\leq
\Ed N_t(\le d,f)+o(\Ed N_t(\le d,f))
\end{split}
\end{equation}
finishing the proof.
\end{proof}
Now, we have all the tools needed to prove Theorem \ref{t:notpowerlaw}
\begin{proof}[Proof of Theorem \ref{t:notpowerlaw}]We only need to prove the case for $\gamma = -1$ and~$F(t)\uparrow\infty$. Otherwise,~$V_t(f)$ is finite almost surely and all vertices are selected infinitely many times with probability one. Given $\e >0$, let $t_k$ be the following deterministic index:
\begin{equation}
t_k = \inf\{ s>0 ; F(s) \ge (1+\e)^k\},
\end{equation}
which exists because we are assuming $F(t)$ goes to infinity. Chebyshev inequality implies then, for every~$\delta>0$,
\begin{equation}
\Pd\left(  \left(\frac{V_{t_k}(f)}{F(t_k)}-1\right) >\delta      \right) \leq \frac{\sum_{s=1}^{t_k}f(s)(1-f(s))}{\delta^2F(t_k)^{2}}\leq \delta^{-2}(1+\varepsilon)^{-k},
\end{equation}
implying that
\[
\frac{V_{t_k}(f)}{F(t_k)}\to 1\quad a.s. \text{ as }k\to\infty.
\]
Combining Lemma~\ref{lemma:var} with the Chebyshev inequality we also get, for any~$\delta>0$,
\begin{equation}
\begin{split}
\Pd\left(  \left(\frac{N_{t_k}(\le d,f)}{F(t_k)}-\frac{\Ed(N_{t_k}(\le d,f))}{F(t_k)}\right) >\delta      \right) \leq \frac{\mathrm{Var}(N_{t_k}(\le d,f))}{\delta^2(1+\varepsilon)^{2k}}\leq C\delta^{-2}(1+\varepsilon)^{-2k},
\end{split}
\end{equation}
and the Borel-Cantelli Lemma together with Proposition~\ref{prop:l1conv} then imply
\begin{equation}
\lim_{t \to \infty } \frac{N_{t_k}(\le d,f)}{F(t_k)} =\lim_{t \to \infty }  \frac{\Ed(N_{t_k}(\le d,f))}{F(t_k)}=0
\end{equation}
almost surely. Now, for~$s\in (t_{k-1},t_k)$ the fact that~$V_t(f)$ and~$F(t)$ are non-decreasing leads to
\begin{equation}
\begin{split}
\frac{V_{s}(f)}{F(s)} &\ge \frac{V_{t_{k-1}}(f)}{(1+\e)^k} \ge \frac{V_{t_{k-1}}(f)}{F(t_{k-1})(1+\e)} > 1-2\e,
\end{split}
\end{equation}
for sufficiently small~$\e$. Therefore, since~$\e$ was chosen arbitrarily,
\begin{equation}
\label{eq:vtvfas}
\frac{V_t(f)}{F(t)}\to 1\quad a.s. \text{ as }t \to \infty,
\end{equation}
implying
\begin{equation}
\lim_{t \to \infty } \frac{N_{t_k}(> d,f)}{F(t_k)} = 1,
\end{equation}
\textit{a.s.} since $N_t(>d,f) = V_t(f) - N_t(\le d,f)$. But observe that~$N_t(>d,f)$ is also non-decreasing in~$t$. Then, for $s \in (t_{k-1},t_k)$ we have that
\begin{equation}
\begin{split}
\frac{N_{s}(> d,f)}{F(s)} &\ge \frac{N_{t_{k-1}}(> d,f)}{(1+\e)^k} \ge \frac{N_{t_{k-1}}(> d,f)}{F(t_{k-1})(1+\e)} > 1-2\e
\end{split}
\end{equation}
for sufficiently small~$\e$, which is enough to conclude that the whole sequence $N_s(>d,f)/F(s)$ converges to $1$ \textit{a.s}, and consequently $N_s(\le d,f)/F(s)$ converges to zero \textit{a.s}.  This, together with~\eqref{eq:vtvfas} concludes the proof.

\end{proof}

%% file: finalremarks.tex

\subsection{Edge-step functions \textit{VS} Affine PA rule} One of the natural generalizations of the preferential attachment rule proposed in \cite{BA99} is known as the \textit{affine preferential attachment} rule. One may introduce a constant $\delta > -1$ on the PA rule (\ref{def:PArule}) so that the probability of a new vertex $v$ connecting to a previous one $u$ is now given by
\begin{equation}
P\left( v \rightarrow u \middle | G\right) = \frac{degree(u)+\delta}{\sum_{w \in G}(degree(w)+\delta)}.
\end{equation}
This slight modification is capable of producing graphs obeying a power-law degree distribution with a tunable exponent lying on $(2,\infty)$. It would be natural to ask the effects on the degree distribution of an \textit{affine} version of our model, since the addition of $\delta$ may lower the degree distribution's tail whereas the edge-step function may lift it. However, the effect of the edge-step function overcomes the presence of $\delta$ in (\ref{def:PArule}) in the long term. We give here some indications of why this is true for $f \in \mathrm{RES}(-\gamma)$, with $\g \in [0,1]$. 

For an affine version of our model, one may start evaluating the identity (\ref{eq:vardeg1}), which is crucial for proof of Theorem~\ref{thm:supE}, to obtain 
\begin{equation*}
\begin{split}
\mathbb{P}\left(\Delta D_t(v) = 1 \middle | \mathcal{F}_t \right)
& = \left(1-\frac{f(t+1)}{2}\right)\frac{\boldsymbol{D_t(v)+\delta}}{t+V_t(f)\delta} - 2\left(1-f(t+1)\right)\frac{(\boldsymbol{D_t(v)+\delta})^2}{(2t+V_t(f)\delta)^2},
\end{split}
\end{equation*}
However, for $f \in \mathrm{RES}(-\gamma)$, with $\g \in [0,1]$, we have by~\eqref{eq:vtvfas} that $V_t(f)=o(t)$. Thus, the above equation is also equal to
\begin{equation*}
\begin{split} 
\left(1-\frac{f(t+1)}{2}\right)\frac{\boldsymbol{D_t(v)+\delta}}{t}(1+o(1)) - 2\left(1-f(t+1)\right)\frac{(\boldsymbol{D_t(v)+\delta})^2}{4t^2}(1+o(1)).
\end{split}
\end{equation*}
The same goes for (\ref{eq:vardeg2}). The same sort of computation also leads to
\begin{equation*}
\mathbb{E}\left[\Delta D_t(v)  \middle | \mathcal{F}_t \right]  
= \left(1-\frac{f(t+1)}{2}\right)\frac{\boldsymbol{D_t(v)+\delta}}{t}(1+o(1)) ,
\end{equation*}
which is \emph{not the case on the usual affine models}. The above identities imply that many of the recursive computations one usually makes regarding these models are not altered by the introduction of the term~$\delta$. In particular, one can use the above recursion and elementary analysis to show that a process with this affine rule produces graphs with the same power-law exponents as the non affine rule ($\delta=0$), though an analogous result to Theorem~\ref{thm:supE} would be significantly more involved. We therefore opted to focus the results in this paper on the case~$\delta=0$.

\subsection{Maximum degree}
Since all the choices are made following the preferential attachment rule, the first vertices on the graph are good candidates for being the ones of highest degree. In this sense, estimates on their degree usually give the exact order of the maximum degree. In this subsection, we estimate the expected degree for the first vertex in order to argue that the presence of edge-step functions may also shape other graph observables, as the maximum degree.

From equations (\ref{eq:vardeg1}) and (\ref{eq:vardeg2}) one may deduce the recurrence relation below for the expected degree of the very first vertex in the graph
\begin{equation}
\Ed \left[ D_{t+1}(1) \right] = \left(1 + \frac{1}{t}-\frac{f(t+1)}{2t}\right)\Ed \left[ D_{t}(1) \right],
\end{equation}
which implies
\begin{equation}
\Ed \left[ D_{t+1}(1) \right] = 2\prod_{s=2}^t\left(1 + \frac{1}{s}-\frac{f(s+1)}{2s}\right) \approx \exp \left \lbrace \sum_{s=2}^t \left(\frac{1}{s}-\frac{f(s+1)}{2s}\right)\right \rbrace.
\end{equation}
If $f$ is taken to be a \textit{regularly} varying function with index of regular variation $\gamma \in (-\infty,0)$, then, by the Representation Theorem and Karamata's Theorem (Corolary~$\ref{thm:repthm}$ and Theorem~$\ref{thm:karamata}$ respectively) it follows that $\sum_{s}^{\infty}f(s)s^{-1}$ is finite and consequently
\[
\Ed \left[ D_{t+1}(1) \right] \approx t.
\]
For a \textit{slowly} varying $f$, the order of $\Ed \left[ D_{t+1}(1) \right]$ depends on
\[
\int_2^t\frac{f(s)}{s}ds.
\]
If $f(s) = \log^{-1}(s)$ we have that
\[
\Ed \left[ D_{t+1}(1) \right] \approx tf(t).
\]
Whereas, for $f(s) = \log^{-2}(s)$, order $t$ is again achieved. The above discussion suggests that even when $f$ is slowly varying, which produces graphs with power-law exponent equal to~$2$, the maximum degree may still be~$f$ dependent. One interesting question would be to determine the precise order of the maximum degree in terms of $f$ when it is taken to be a slowly varying function.

%% file: appendix.tex
\section{Important results on Regularly Varying Functions}\label{app:rvf}
In this appendix, we collect some results regarding regularly varying functions that will be useful throughout the paper, as well as providing a proof for an earlier lemma. 

\begin{corollary}\label{cor:lFt} Let $\ell$ be a slowly varying function. Then,
	\begin{equation}
	\lim_{t \rightarrow \infty} \frac{\ell(t)}{\sum_{s=1}^t\ell(s)s^{-1}} = 0
	\end{equation}
\end{corollary}
\begin{proof} Since $s^{-1}\ell(s)$ is a RV function which is eventually monotone, we may bound the sum by the integral. Now, by Theorem $1.5.2$ of~\cite{regvarbook}, for a fixed small $\e$, we know that
	\[
	\lim_{t \to \infty}\frac{\ell(xt)}{\ell(t)} =1
	\]
uniformly for $x \in [\e,1]$. Therefore, for large enough $t$
	\begin{equation}
	\begin{split}
	\ell^{-1}(t)\int_{1}^t\frac{\ell(s)ds}{s} & \ge \int_{\e}^1\frac{\ell(tx)dx}{\ell(t)x} \ge -(1-\delta)\log\e,
	\end{split}
	\end{equation}
	for some small $\delta$. This proves the desired result. 
\end{proof}
The three following results are used throughout the paper.
\begin{corollary}[Representation theorem - Theorem~$1.4.1$ of~\cite{regvarbook}]\label{thm:repthm} Let $f$ be a continuous regularly varying function with index of regular variation $\gamma$. Then, there exists a slowly varying function $\ell$ such that
	\begin{equation}
		f(t) = t^{\gamma}\ell(t),
	\end{equation}
	for all $t$ in the domain of $f$.
\end{corollary}
\begin{corollary}\label{cor:a3}Let $f$ be a continuous regularly varying function with index of regular variation $\gamma <0$. Then,
	\begin{equation}
		f(x) \to 0,
	\end{equation}
	as $x$ tends to infinity. Moreover, if $\ell$ is a slowly varying function, then for every $\varepsilon >0$
	\begin{equation}
	x^{-\varepsilon} \ell(x) \to 0 \text{ and } x^{\varepsilon}\ell(x) \to \infty
	\end{equation}
\begin{proof}
Comes as a straightforward application of Theorem~$1.3.1$ of~\cite{regvarbook} and~Corollary~\ref{thm:repthm}.
\end{proof}
\end{corollary}
\begin{theorem}[Karamata's theorem - Proposition~$1.5.8$ of~\cite{regvarbook}]\label{thm:karamata} Let $\ell$ be a continuous slowly varying function and locally bounded in $[x_0, \infty)$ for some $x_0 \ge 0$. Then
	\begin{itemize}
		\item[(a)] for $\alpha > -1$
		\begin{equation}
			\int_{x_0}^{x}t^{\alpha}\ell(t)dt \sim \frac{x^{1+\alpha}\ell(x)}{1+\alpha}. 
		\end{equation}
		
		\item[(b)] for $\alpha < -1$
		\begin{equation}
		\int_{x}^{\infty}t^{\alpha}\ell(t)dt \sim \frac{x^{1+\alpha}\ell(x)}{1+\alpha}.
		\end{equation}
	\end{itemize}
\end{theorem}
We finish this section with the proof of an earlier lemma.
\begin{proof}[Proof of Lemma~$\ref{lemma:rvint}$]
	(i) By Potter's Theorem (Theorem~$1.5.6$ of\cite{regvarbook}), if~$\ell$ is slowly varying then for every~$\delta>0$ there exists~$M>0$ such that
	\begin{equation}
	\label{eq:potterthmbound}
	\frac{\ell(x)}{\ell(y)}\leq 2\max\left\{   \frac{x^\delta}{y^\delta}, \frac{y^\delta}{x^\delta}\right\}
	\end{equation}
	for every~$x,y>M$. We have
	\begin{align}
	\nonumber
	\int_{0}^{1} \left| \frac{\ell(ut)}{\ell(t)}-1 \right|u^{-\gamma}\mathrm{d} u
	\nonumber&=
	\int_{0}^{\frac{M}{t}}  \left| \frac{\ell(ut)}{\ell(t)}-1 \right|u^{-\gamma}\mathrm{d} u
	+
	\int_{\frac{M}{t}}^{1} \left| \frac{\ell(ut)}{\ell(t)}-1 \right|u^{-\gamma}\mathrm{d} u.
	\end{align}
	We then obtain
	\[
	\int_{0}^{\frac{M}{t}}  \left| \frac{\ell(ut)}{\ell(t)}-1 \right|u^{-\gamma}\mathrm{d} u\leq  \left( \frac{\sup_{y\in[0,M]}\ell(y)}{\ell(t)}-1 \right)\frac{M^{1-\gamma}}{t^{1-\gamma}(1-\gamma)}\xrightarrow{t\to\infty} 0,
	\]
	by Corollary~$\ref{cor:a3}$. Choosing~$\delta<1-\gamma$ in~\eqref{eq:potterthmbound}, we see that
	\begin{align}
	\nonumber
	\int_{0}^{1} \left| \frac{\ell(ut)}{\ell(t)}-1 \right|u^{-\gamma}\mathbb{1}\{u\geq M/t\}\mathrm{d} u
	\leq 
	\int_{0}^{1} \left(2 \max\{u^{-\delta},u^{\delta}\}-1 \right)u^{-\gamma}\mathrm{d} u<\infty,
	\end{align}
	and therefore the LHS of the above equation tends to~$0$ by the dominated convergence Theorem. This and another elementary application of Corollary~$\ref{cor:a3}$ finish the proof of item~(i).
	
	(ii) We have
	\begin{align}
	\nonumber
	\left| \sum_{k=1}^{t}\ell(k)k^{-\gamma}  -\frac{t^{1-\gamma}\ell(t)}{1-\gamma}     \right|
	&\leq 
	\left| \sum_{k=1}^{t}\ell(k)k^{-\gamma}  -\int_{0}^{t}\ell(s)s^{-\gamma}\mathrm{d} s     \right|
	+
	\left| \int_{0}^{t}\ell(s)s^{-\gamma}\mathrm{d}s   - \ell(t)\cdot\int_{0}^{t} s^{-\gamma}\mathrm{d}s  \right|
	\\ \label{eq:hbound1}
	&\leq C +  \left| \int_{0}^{t}s^{-\gamma}(\ell(s)-\ell(t))\mathrm{d}s  \right|,
	\end{align}
	since~$\ell(s)s^{-\gamma}$ is eventually monotone decreasing. Dividing both sides by~$t^{1-\gamma}\ell(t)$ and making the substitution~$u=st^{-1}$ in the integral gives the result.
\end{proof}

\section{Martingales concentration inequalities}\label{a:martingconc}
For the sake of completeness we state here two useful concentration inequalities for martingales which are used throughout the paper.
\begin{theorem}[Azuma-H\"{o}ffeding Inequality - \cite{CLBook}]\label{t:azuma} Let $(M_n,\mathcal{F})_{n \ge 1}$ be a (super)martingale satisfying
\[ 
\lvert M_{i+1} - M_i \rvert \le a_i
\]
Then, for all $\l > 0 $ we have
\[
\P \left( M_n - M_0 > \l \right) \le \exp\left( -\frac{\l^2}{\sum_{i=1}^n a_i^2} \right).
\]
\end{theorem}

\begin{theorem}[Freedman's Inequality - \cite{F75}]\label{teo:freedman} Let $(M_n, \mathcal{F}_n)_{n \ge 1}$ be a (super)martingale. Write 
\begin{equation}\label{def:quadvar}
W_n := \sum_{k=1}^{n-1} \mathbb{E} \left[(M_{k+1}-M_k)^2\middle|\mathcal{F}_k \right]
\end{equation}
and suppose that $M_0 = 0$ and 
\[ 
\lvert M_{k+1} - M_k \rvert \le R, \textit{ for all }k.
\]
Then, for all $\l > 0 $ we have
\[
\P \left(   M_n \ge \l, W_n \le \s^2, \textit{ for some }n\right) \le \exp\left( -\frac{\l^2}{2\s^2 + 2R\l /3} \right). 
\]
\end{theorem}